\documentclass[amsart,11pt]{article}
\usepackage{color,graphicx,latexsym,amssymb,amsmath,amsfonts,amsthm,float}
\usepackage{url}
\usepackage{hyperref,mathrsfs}
\hypersetup{colorlinks}
\usepackage{authblk}

\theoremstyle{plain}

\newtheorem{theorem}{Theorem}[section]

\newtheorem{lemma}[theorem]{Lemma}

\newtheorem{conjecture}[theorem]{Conjecture}
\newtheorem{corollary}[theorem]{Corollary}
\newtheorem{prop}[theorem]{Proposition}

\theoremstyle{definition}
\newtheorem{remark}{Remark}
\newtheorem{definition}{Definition}

\newcommand{\Z}{\mathbb Z}
\newcommand{\E}{\mathbb E}
\newcommand{\PP}{\mathbb P}

\newcommand{\comp}{c}

\newcommand{\e}{\epsilon}

\definecolor{db}{rgb}{0.1,0,0.75}
\definecolor{lm}{cmyk}{0 ,1,0,0}

\newcommand{\R}{\mathbb R}
\newcommand{\ft}{\mathcal G_t}
\newcommand{\W}{\text{Wait}}

\newcommand{\m}{\mathcal{G}}
\newcommand{\n}{\mathcal{G}_1}
\newcommand{\mm}{\mathcal{G}_2}

\newcommand{\h}{\operatorname{H}}
\newcommand{\base}{\operatorname{Base}}
\newcommand{\zeros}{\operatorname{Zeros}}
\newcommand{\Height}{\operatorname{Height}}
\newcommand{\height}{\operatorname{Height}}
\newcommand{\volume}{\tau}
\newcommand{\returns}{\operatorname{Returns}}

\newcommand{\badone}{\mathcal{B}_1}
\newcommand{\badi}{\mathcal{B}_i}
\newcommand{\badtwo}{\mathcal{B}_2}
\newcommand{\badthree}{\mathcal{B}_3}

\newcommand{\badzero}{\mathcal{B}_0}
\newcommand{\even}{\tilde R^x}
\renewcommand{\star}{\mathcal{E}}

\newcommand{\regular}{\operatorname{Regular}}

\renewcommand{\P}{\mathbb{P}}
\newcommand{\eps}{\varepsilon}

\newenvironment{pfoflem}[1]
{\par\vskip2\parsep\noindent{\sc Proof of Lemma\ #1. }}{{\hfill
$\Box$}
\par\vskip2\parsep}

\newcommand{\pr}{\mathbb P}
\newcommand{\N}{\mathbb N}
\newcommand{\old}[1]{}

\renewcommand{\P}{\mathbb{P}}

\title{Oil and water: a two-type internal aggregation model}

\author[1]{Elisabetta Candellero\thanks{E.Candellero@warwick.ac.uk. \ Supported by the Austrian Science Fund (FWF): E-1503W01230, %Doctoral Program ``Discrete Mathematics'' 
and by a Marie Curie Career Integration Grant PCIG09-GA2011-293619.}}
\author[2]{Shirshendu Ganguly\thanks{sganguly@math.washington.edu.\ Partially supported by NSF grant  DMS-0847661.}}
\author[3]{Christopher Hoffman\thanks{hoffman@math.washington.edu. \ Supported by NSF grant DMS-1308645 and NSA grant H98230-13-1-0827.}}
\author[4]{Lionel Levine\thanks{\url{http://www.math.cornell.edu/\~levine}. Partially supported by NSF grant DMS-1243606.}}
 \affil[1]{University of Warwick}
 \affil[2,3]{University of Washington}
 \affil[4]{Cornell University}

\date{August 4, 2014}

\DeclareMathOperator{\Var}{Var} 

\begin{document}
\maketitle

\begin{abstract}
We introduce a two-type internal DLA model which is an example of a non-unary abelian network. Starting with $n$ ``oil'' and $n$ ``water'' particles at the origin, the particles diffuse in $\Z$ according to the following rule: whenever some site $x \in \Z$ has at least $1$ oil and at least $1$ water particle present, it {\bf fires} by sending $1$ oil particle and $1$ water particle each to an independent random neighbor $x \pm 1$. 
Firing continues until every site has at most one type of particles. 
We establish the correct order for several statistics of this model and identify the scaling limit under assumption of existence.
\end{abstract}
%\newpage
%\tableofcontents
%\newpage

\section{Introduction and Main Results}\label{mainresults}
%!TEX root = oil-and-water.tex

We investigate a new interacting particle system on $\mathbb{Z}$ that can be considered as a model of \emph{mutual diffusion}. Two particle species, called for convenience as \emph{oil} and \emph{water}, diffuse on $\Z$ until there is no site that has both an oil and water particle.
We start with $n$ oil and $n$ water particles at the origin. At each discrete time step, if at least 1 oil particle and at least 1 water particle are present at $x \in \Z$ then $x$ \textbf{fires} by sending $1$ oil particle and $1$ water particle each to an independent random neighbor $x \pm 1$ with equal probability.  The system fixates when no more firing is possible, that is, when every site has particles of at most one type. 

How many firings are required to reach fixation? How far is the typical particle from the origin upon fixation? Our main results address these two questions.

%\subsection{Main Results}\label{mainresults}

\begin{definition}\label{defo}For $x\in \mathbb{Z}$ let $u(x)$ be the total number of times $x$ fires before fixation.  The random function $u : \Z \to \N$ is called the \textbf{odometer} of the process.
\end{definition} 

%\begin{remark}
In the above informal description we have assumed that all sites fire in parallel in discrete time, but in fact this system has an \textbf{abelian property}: the distribution of the odometer and of the particles upon fixation do not depend on the order of firings (Lemma~\ref{l.abelian}).
%\end{remark}

Our first result concerns the order of magnitude of the odometer. 

\begin{theorem}\label{lemmaJ}
There exist  positive numbers $\e,c,C$ such that for large enough $n,$
\begin{itemize}
\item [i.]  $$\pr \left( \sup_{x\in \mathbb{Z}}u(x) >C n^{4/3}\right) < e^{-n^\e}$$ 
\item [ii.] $$\pr \left(\inf_{x:|x|\le c n^{1/3}}u(x) < c n^{4/3} \right) <   e^{-n^\e}.$$
\end{itemize}

\end{theorem}

The next result shows that most particles do not travel very far: all but a vanishing fraction of the particles at the end of the process are supported on an interval of length $n^{\alpha}$, for any exponent $\alpha > 1/3$.  Formally, for $r>0$ let $F(r)$ be the number of particles that fixate outside the interval $[-r,r]$.

\begin{theorem}\label{spread1}
For sufficiently small $\epsilon>0$ there exists $\delta>0$ such that 
	\[ \P\left( F  \big(n^{\frac13 + \eps} \big) > n^{1- \frac{\epsilon}{2}} \right) < e^{-n^{\delta}}. \]
\end{theorem}

 Theorem \ref{lemmaJ}  shows that the odometer function scales like $n^{4/3}$.
Theorem \ref{spread1} motivates the conjecture that the proper scaling factor in the spatial direction should be $n^{1/3}$. We conjecture that the scaling limit of the odometer exists, and under the assumption of existence we identify the limiting function.

Let  $\tilde u := \E u$. 
\begin{conjecture}\label{ass1} 
\begin{itemize}
\item[(i)] For any $\delta>0$
\begin{equation}\label{eq:assumption}
\pr \left ( \sup_{x \in \Z} \left | \frac{u(x)-\tilde{u}(x)}{n^{4/3}}\right | >\delta \right )\to 0,  \ \textnormal{ as }n\to \infty.
\end{equation}
\item[(ii)] There is a function $w: \R \to \R$ such that
\begin{equation}\label{eq:w}
\frac{\tilde{u}(\lfloor{n^{1/3}\xi}\rfloor)}{n^{4/3}} \, \to \, w(\xi),
\end{equation}
uniformly in $\xi$.
\end{itemize}
\end{conjecture}

Simulations support Conjecture~\ref{ass1}, as shown in Figure~\ref{f.1d}. Conditionally on Conjecture \ref{ass1}, the following result identifies the limit $w(x)$ exactly.

\begin{theorem}\label{conditionalscalingthm} Assuming Conjecture \ref{ass1}, the function $w$ appearing in \eqref{eq:w} must equal
\begin{equation} \label{conjecturedscalinglimit1}
w(x) = \begin{cases} \frac{1}{72\pi}\left ((18\pi)^{1/3} -|x| \right)^4, & |x|< (18\pi)^{1/3} \\
				0 & |x| \geq (18\pi)^{1/3}. \end{cases}
\end{equation}
\end{theorem}

\begin{figure}
\centering
\includegraphics[width=\textwidth]{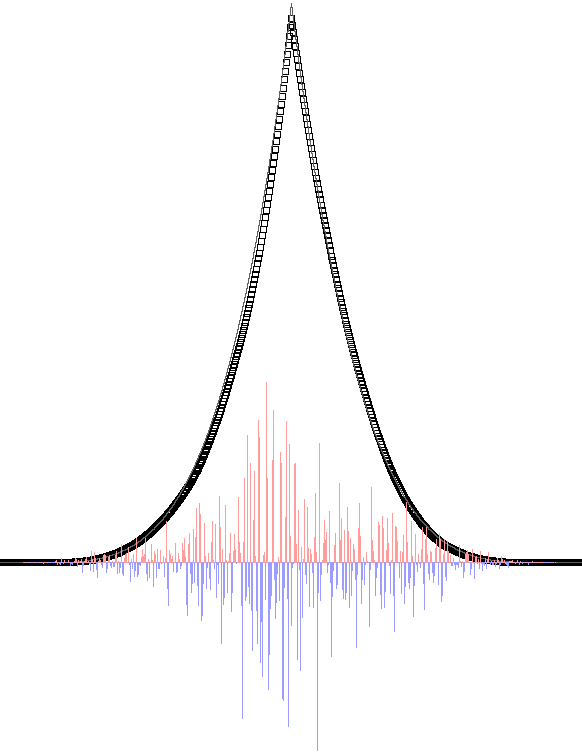}
\caption{Graph of the odometer function of the oil and water model in $\Z$ with $n=360000$ particles of each type started at the origin: for each $x \in \Z$ a box is drawn centered at $(x,u(x))$ where $u(x)$ is the number of oil-water pairs fired from $x$.  The curve $f(x)= \frac{1}{72\pi} ((18\pi n)^{1/3} - |x|)^4$ appears in gray. Red and blue vertical bars represent the final configuration of oil and water particles respectively; the height of the bar is proportional to the number of particles.
\label{f.1d}}
\end{figure}

\newpage
\text{ }

The oil and water model can be defined on any graph and in particular on higher-dimensional lattices. Figure~\ref{f.2d} shows an oil and water configuration in $\Z^2$. In Section~\ref{s.open} we conjecture the relevant exponents in $\Z^d$ for $d \geq 2$.

\begin{figure}
\centering
\includegraphics[width=4in]{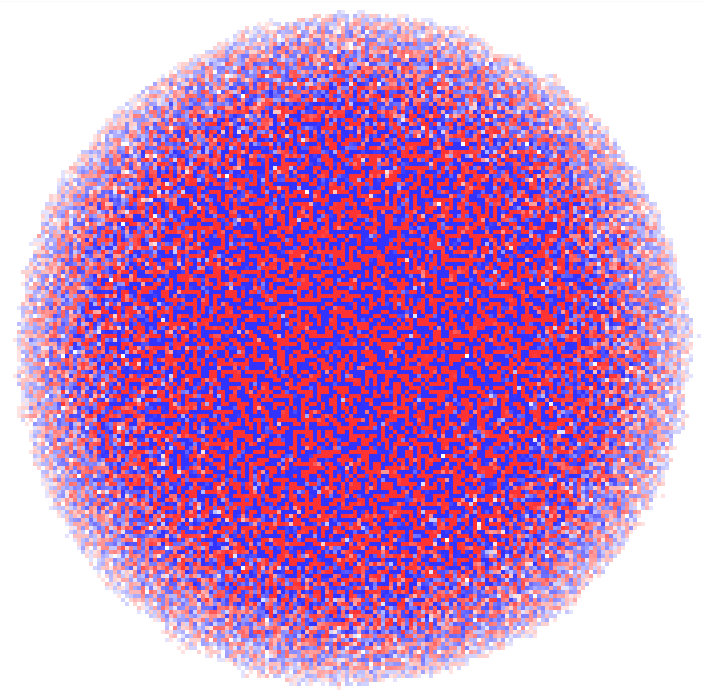}
\caption{Oil and water in $\Z^2$ with $n=2^{22}$ particles of each type started at the origin. Each site where particles stop is shaded red or blue according to whether oil or water particles stopped there. The intensity of the shade indicates the number of particles. We believe that the limit shape is a disk of radius of order $n^{1/4}$.
\label{f.2d}}
\end{figure}

%To put the results in context, we briefly describe some related models.

\subsection{Related models: internal DLA and abelian networks}

In \emph{internal DLA}, each of $n$ particles started at the origin in $\Z^d$ performs a simple random walk until reaching an unoccupied site. The resulting random set of $n$ occupied sites is close to a Euclidean ball \cite{LBG}.  Internal DLA is one of several models known to have an \emph{abelian property} according to which the order of moves does not affect the final outcome.

Dhar \cite{Dhar} proposed certain collections of communicating finite automata as a broader class of models with this property.  
Until recently the only examples studied in this class have been \emph{unary networks} (or their ``block renormalizations'' as proposed in \cite{Dhar}).  Informally, a unary network is a system of local rules for moving \emph{indistinguishable} particles on a graph, whereas a non-unary network has different types of particles.  It is not as easy to construct non-unary examples with an abelian property, but they exist. Alcaraz, Pyatov and Rittenberg studied a class of non-unary examples which they termed \emph{two-component sandpile models} \cite{2comp}, and asked whether there is a nontrivial example with two particle species such that the total number of particles of each type is conserved.  
Oil and water has this conservation property, but differs from the two-component sandpile models in that any number of particles of a single type may accumulate at the same vertex and be unable to fire.

Bond and Levine \cite{abnet} developed Dhar's idea into a theory of \emph{abelian networks} and proposed two non-unary examples, \emph{oil and water} and \emph{abelian mobile agents}. Can such models exhibit behavior that is ``truly different'' from their unary cousins? This is the question that motivates the present paper.
Theorem~\ref{spread1} shows that oil and water has an entirely different behavior than internal DLA: all but a vanishing fraction of the $2n$ particles started at the origin stop within distance $n^{\frac13 + \eps}$ (versus $n$ for internal DLA).

In contrast to internal DLA where there is now a detailed picture of the fluctuations in all dimensions  \cite{AG1,AG2,AG3,JLS1,JLS2,JLS3}, there is not even a limiting shape theorem yet for oil and water. Simulations in $\Z^d$ indicate a spherical limit shape (Figure~\ref{f.2d}) with radius of order $n^{1/(d+2)}$.

%Our model has the mathematical motivation just described, but it is natural to ask for a physical interpretation. What we have is a model of two substances which can each diffuse only in the presence of the other. Does this system have a physical instantiation? 
%The name ``oil and water'' suggests the metaphor of an \emph{emulsion} (a mixture of two or more immiscible liquids). However, its behavior does not really resemble that of a true emulsion: emulsions have large-scale fractal droplets \cite{emulsion} which are not apparent in simulations of our model. Ising-type models are a common choice for modeling emulsions \cite{microemulsion}. Our model is more short-range: it has no interactions between neighboring sites, only same-site interactions.

\subsection{Main ideas behind the proofs.} \label{miop}
In this section we present informally the key ideas behind the proofs.  We start with a definition. 
\begin{definition}\label{particlenotation} For any $x\in \mathbb{Z}$ and $t\in \{0,1,\ldots \infty\}$ let $\eta_1(x,t)$ and $\eta_2(x,t)$ represent the number of oil and water particles respectively, at position $x$ at time $t$.
\end{definition}

From now on, let $$P_t = \sum_{x \in \Z} \min(\eta_1(x,t),\eta_2(x,t))$$ be the number of co-located oil-water pairs at time $t$. Furthermore, let $x_t$ denote the left most site with a co-located oil-water pair at time $t$. 

The process is run in phases: we start by firing a pair at the origin and, inductively, at each time $t>0$ we locate the vertex $x_t$ and let it fire.
More precisely, unless otherwise stated, in the following we will make the tacit assumption that at every time $t$ the process fires one pair from $x_t$. 
By definition, $P_0=n$ and the process stops at the first time $\tau$ for which $P_\tau=0$.  
Now denote by $l_t = \eta_1(x_t-1,t)-\eta_2(x_t-1,t)$ the excess of oil particles over water particles at the left neighbor of $x_t$ and by $r_t = \eta_1(x_t+1,t)-\eta_2(x_t+1,t)$ the excess at the right neighbor.

Conditional upon knowing the process up to time $t$, the random variable 
\begin{equation}\label{spl}
Z_{t}:=P_{t+1}-P_t
\end{equation}
can have four possible distributions. To distinguish them, we denote the conditional $Z_{t+1}$ by the random variables $\xi_1,\xi_2,\xi_3,\xi_4$ which are described in the following table:
\begin{center}
\begin{tabular}{|c|c|c|c|c|}
  \hline
$i$&$\PP(\xi_i=-1)$&$\PP(\xi_i=0)$&$\PP(\xi_i=1)$ & used when \\%[1.1ex]
\hline
$1$& $1/4$&$1/2$&$1/4$ & $l_t r_t < 0$ ~ \\%[1.1ex]
\hline
$2$ &$1/4$&$3/4$&$0$ & exactly one of $l_t, r_t$  is $0$ \\%[1.1ex]
\hline
$3$ & $1/2$&$1/2$&$0$ & $l_t = r_t =0$ \\%[1.1ex] %\cline{2-2}
\hline
$4$&$0$&$1$&$0$ &$ l_t r_t > 0$ \\%[1.1ex]
  \hline
\end{tabular}
\end{center}
Note that $\xi_1$ has mean zero, whereas $\xi_2$ and $\xi_3$ have negative means, and $\xi_4$ is degenerate at $0$. Since all such expected values are less than or equal to zero, 
\begin{center}
\textbf{$P_t$ is a supermartingale.}  
\end{center}
The main idea of the proof for the upper bound in Theorem \ref{lemmaJ} is as follows: Define the auxiliary random variables 
\begin{equation}\label{eq:def_N_i}
N_i(t) = \# \{s \leq t \,:\, Z_s \stackrel{d}{=} \xi_i\}. 
\end{equation}
It is clear from the definition that 
$$
\sum_{i=1}^{\tau}Z_i=P_\tau-P_0=-n
$$
where $\tau$ is the stopping time of the process where all oil and water have
been separated ($P_\tau = 0$). 
Now, informally $\sum_{i=1}^{\tau}Z_i$ is the sum of $N_1(\tau)+N_2(\tau)$ variables with mean at most $-1/4$ and $N_3(\tau)$ variables with mean $0$. 

Because of the negative and zero drifts respectively, the sum of the negative mean variables is of order $O\bigl (-(N_1(\tau)+N_2(\tau))\bigr )$, whereas the sum of the zero mean variables is roughly $O(\sqrt{N_{3}(\tau)})$ (by square root fluctuations of the symmetric random walk).
Thus roughly $$
\sum_{i=1}^{\tau}Z_i\le -N_1(\tau)-N_2(\tau)+\sqrt{N_{3}(\tau)}.$$
We then argue by contradiction: conditional upon the event that the odometer is ``very large'' somewhere (i.e.\ larger than $ Cn^{4/3}$ for a suitable big constant $C$), we show that 
it is likely that $ N_1(\tau)+  N_2(\tau)$ is ``sufficiently large''  compared to $\sqrt{N_3(\tau)}$.  This in turn implies that 
$$\sum_{i=1}^{\tau}Z_i< -n$$ which is a contradiction.
%
%since the 
%Then the Azuma-Hoeffding inequality at time $\tau$ tells us this happens with small probability.  Thus the probability that $u(x)$ is somewhere large must be small.

To prove the lower bound we first establish a gradient bound on the odometer using the upper bound. In other words, we show that there exists a constant $c$ such that, with high probability, for  $x\in[0,cn^{1/3}]$ we have
$$u(x)-u(x+1)\ge n/2.$$
This in turn implies that with high probability $$u(0)=\Omega(n^{4/3}),$$
since $$0\le u(\lfloor{cn^{1/3}}\rfloor)\le u(0)-\frac{n}{2}\lfloor{cn^{1/3}}\rfloor.$$ 
Theorem \ref{spread1} follows from the proof of the lower bound.

Lastly we discuss the proof of Theorem \ref{conditionalscalingthm}. Clearly by symmetry of the process about the origin the function $w(\cdot)$ is symmetric as well.  We  show that Conjecture \ref{ass1} implies that the limiting function is smooth on the positive real axis and in particular satisfies
 $$w''(x)=\sqrt{\frac{2}{\pi}w(x)}$$
with certain boundary conditions. At this point, Theorem \ref{conditionalscalingthm} follows by identifying a solution of the above boundary value problem, and using uniqueness of the solution.

\subsection{Outline}

This article is structured as follows.
In Section \ref{sdes} we give a rigorous definition of the model. Furthermore, in order to facilitate our proofs, we define two different versions of the process and prove that they both terminate in finite time with probability one. We then construct a coupling between the two versions, which in particular implies that they terminate with the same odometer function and the same distribution of particles. 

In Section \ref{section:apriori} we show a polynomial bound on the stopping time $\tau$. We start by showing that the number of pairs $P_t$ can be stochastically dominated by a certain lazy random walk with long holding times, started and reflected at $n$. 
In order to get a rough upper bound on $\tau$, it suffices to bound the hitting time of zero for this walk, which we show is of  order $n^4$.
Consequently, we get direct polynomial bounds on the support and the maximum of the odometer function.
%Note that this also gives us a bound on the maximum of the odometer function.
%for the stopping time (as a function of the initial number of particles).
We then  improve the last bound, proving the upper bound in Theorem \ref{lemmaJ} in Section \ref{upperbound}.

Section \ref{lowerbound} is devoted to  proving  the lower bound in Theorem \ref{lemmaJ}. As a byproduct of the proof, Theorem \ref{spread1} follows.

 In Section \ref{prop} we prove the conditional Theorem \ref{conditionalscalingthm}. 
Section \ref{s.open} consists of open questions and a conjecture about the process on higher dimensional lattices.

\section{Rigorous definition of the model}\label{sdes}
%!TEX root =oil-and-water.tex

As the underlying randomness for our model we will take a countable family of independent random variables 
\begin{equation}\label{set1}
\omega = (X^x_k,Y^x_k)_{x \in \Z, k \in \N}
\end{equation}
 with $\P(X^x_k=\pm 1)=P(Y^x_k=\pm 1)=\frac12$.  For each $x \in \Z$ the sequences $(X^x_k)_{k \geq 1}$ and $(Y^x_k)_{k \geq 1}$ are called the \emph{stacks at $x$}. 
Denote by $\Omega$ the set of all stacks $\omega$.  
 
 The stacks will have the following interpretation (described formally below): On the $k$th firing from site $x$, an oil particle steps from $x$ to $x+X^x_k$ and a water particle steps from $x$ to $x+Y^x_k$.

For any value $K\in \mathbb{N}$, a \emph{firing sequence} is a sequence of vertices $s = (x_0,\ldots,x_{K-1})$ with each $x_k \in \Z$. 
Recall from Definition \ref{particlenotation} that the random variables $\eta_i(x,t)$ represent the number of particles of type $i$ (type $1$ are \emph{oil} and type $2$ are \emph{water} particles) at location $x$ at time $t$. Given such a sequence and an initial state $\eta_1(\cdot,0),\eta_2(\cdot,0)$ we define the oil and water process $(\eta_1(\cdot,k),\eta_2(\cdot,k))_{k=0}^K$ inductively by
	\begin{align}\label{kermit} \eta_1(\cdot,k+1)  = \eta_1(\cdot,k)  - \delta(x_k) + \delta(x_k+X^{x_k}_{i_k}) \\ 
	     \eta_2(\cdot,k+1) = \eta_2(\cdot,k) - \delta(x_k) + \delta(x_k+Y^{x_k}_{i_k}) 
	     \label{thefrog}\end{align}
where $i_k = \# \{j \leq k \,:\, x_j = x_k \}$.
Here $\delta(x)$ is the function taking value $1$ at $x$ and $0$ elsewhere.

\subsection{Least action principle and abelian property}

\begin{definition}
Let $s=(x_0,\ldots,x_{K-1})$ be a firing sequence. We say that $s$ is \emph{legal} for $(\eta_1(\cdot,0),\eta_2(\cdot,0))$ if 
	\[ \min(\eta_1(x_k,k), \eta_2(x_k,k)) \geq 1, \]
for all $k$, $0 \leq k \leq K-1$. We say that $s$ is \emph{complete} for $(\eta_1(\cdot,0),\eta_2(\cdot,0))$ if the final configuration $(\eta_1(\cdot,K),\eta_2(\cdot,K))$ satisfies
	\[ \min(\eta_1(x,K),\eta_2(x,K)) \leq 0, \]
for all $x \in \Z$.
\end{definition}
Making Definition \ref{defo} precise we define the \emph{odometer} of a firing sequence $s$ to be the function $u_s :\Z \to \N$ given by
\begin{equation}\label{interodo}
 u_s(x) = \# \{0 \leq k <K \,:\, x_k=x \}.
 \end{equation}

How does $u_s$ depend on $s$? The least action principle for abelian networks addresses this question.

\begin{lemma} \label{l.LAP}
{\em (Least Action Principle, \cite{abnet})}
Let $s$ and $s'$ be firing sequences. If $s$ is legal for $(\eta_1(\cdot,0),\eta_2(\cdot,0))$ and $s'$ is complete for $(\eta_1(\cdot,0),\eta_2(\cdot,0))$, then $u_s(x) \leq u_{s'}(x)$ for all $x \in \Z$.
\end{lemma}

We remark that this statement holds pointwise for any stacks $\omega$, even if $s$ and $s'$ are chosen by an adversary who knows the stacks.

In this paper we will not use the full strength of Lemma~\ref{l.LAP}. Only the following corollaries will be used.

\begin{lemma}\label{l.abelian}
{\em (Abelian property \cite{abnet})}
For fixed stacks $\omega$ and fixed initial state $(\eta_1(\cdot,0),\eta_2(\cdot,0))$,
\begin{enumerate}
\item[(i)] If there is a complete firing sequence of length $K$ then every legal firing sequence has length $\leq K$.

\item[(ii)] If firing sequences $s$ and $s'$ are both legal and complete, then $u_s = u_{s'}$.

\item[(iii)] Any two legal and complete firing sequences result in the same final state $(\eta_1(\cdot,K),\eta_2(\cdot,K))$.
\end{enumerate}
\end{lemma}

\subsection{Leftmost convention}
We fix the initial state $$\eta_1(\cdot,0)=\eta_2(\cdot,0) = n \delta(0)$$ of $n$ oil and $n$ water particles at the origin and none elsewhere.
By Lemma~\ref{l.abelian}, for fixed stacks $\omega$, the final state does not depend on the choice of legal and complete firing sequence. However, many of our lemmas will require choosing a particular sequence. Unless otherwise specified, we always fire the \emph{leftmost legal site}:
\begin{equation}\label{eq:xt}
 x_k = \min \{ x \in \Z \,:\, \min(\eta_1(x,k), \eta_2(x,k)) \geq 1 \}. 
\end{equation}
	
If the set on the right side is empty, then $(\eta_1(\cdot,k), \eta_2(\cdot,k))$ is the final state and $x_k$ is undefined.  We denote by $\tau \in \N \cup \{\infty\}$ the first time at which this happens:
	\begin{equation}\label{stoptimedef}
	 \tau = \min \{ k \geq 0 \,:\, \min(\eta_1(x,k), \eta_2(x,k)) = 0 \text{ for all } x \in \Z \}.
	 \end{equation}
   
Note that everything so far holds pointwise in the stacks $\omega$. The first statement involving probability is that $\tau<\infty$, almost surely in $\omega$, and  we postpone its proof to Section \ref{section:apriori}.

\subsection{Merged stacks}\label{sect:version2}

We describe another process $(\eta'_1(\cdot,\cdot),\eta'_2(\cdot,\cdot))$ which has the same law as the process $(\eta_1(\cdot,\cdot),\eta_2(\cdot,\cdot))$ defined in Definition \ref{particlenotation}.   Throughout the rest of the article we will refer to the process $(\eta_1(\cdot,\cdot),\eta_2(\cdot,\cdot))$ as \textbf{Version} $1$ and the modified process as \textbf{Version} $2$.\\

In \textbf{Version} $2$, analogous to \eqref{set1} the source of randomness is a set of independent variables (modified stacks) 
\begin{equation}\label{stack22}
\omega' = (X^x_i, Y^x_i, \bar{X}^x_i, \bar{Y}^x_i)_{x \in 3\Z, i\in \N}. 
\end{equation}
 with $\P(X^x_i=\pm 1)=P(Y^x_i=\pm 1)=\P(\bar{X}^x_i=\pm 1)=\P(\bar{Y}^x_i=\pm 1)\frac12$.
 Denote by $\Omega'$ the set of all stacks $\omega'.$ 
Note that stacks are located only at every $x \in 3\Z$. 
 
Informally, the $i$th firing from $x \in 3\Z$ uses moves $X^x_i$ and $Y^x_i$ as before, but the $i$th firing from the \emph{set} $\{x-1,x+1\}$ uses moves $\pm \bar{X}^x_i$ and $\pm \bar{Y}^x_i$ according to whether the firing was from $x-1$ or $x+1$ respectively.  We refer to $(\bar{X}^x_i, \bar{Y}^x_i)_{i \geq 1}$ as the \emph{merged stacks} of $x-1$ and $x+1$.

Formally, given a firing sequence $s =(x_0, \ldots, x_{K-1})$, we define $(\eta'_1(\cdot,k),\eta'_2(\cdot,k))$ inductively as follows. If $x_k \in 3\Z$ then
	\begin{align*} \eta'_1(\cdot,k+1) = \eta'_1(\cdot,k) - \delta(x_k) + \delta(x_k+X^{x_k}_{i_k}) \\ 
	                \eta'_2(\cdot,k+1) = \eta'_2(\cdot,k+1) - \delta(x_k) + \delta(x_k+Y^{x_k}_{i_k}) \end{align*}
where $i_k = \# \{j \leq k \,:\, x_j = x_k \}$.
If $x_k \in 3\Z \pm 1$ then
	\begin{align*} \eta'_1(\cdot,k+1) = \eta'_1(\cdot,k) - \delta(x_k) + \delta(x_k \mp \bar{X}^{x_k \mp 1}_{i'_k}) \\ 
	\eta'_2(\cdot,k+1) = \eta'_2(\cdot,k+1) - \delta(x_k) + \delta(x_k \mp \bar{Y}^{x_k \mp 1}_{i'_k}) \end{align*}
where $i'_k = \# \{j \leq k \,:\, x_j \in \{x_k,x_k \mp 2\} \}$.

To compare the modified process to the original we use the following Proposition.

\begin{prop}
\label{p.cards}
Let $(Z_i)_{i \in I}$ be independent uniform $\pm 1$-valued random variables indexed by a countable set $I$. Let $i_1,i_2,\ldots \in I$ be a sequence of distinct random indices and $\xi_1,\xi_2,\ldots$ a sequence of $\pm 1$-valued random variables such that for all $k \geq 1$ both $i_k$ and $\xi_k$ are measureable with respect to 
%\begin{equation}  
$\mathcal{F}_{k-1} := \sigma (Z_{i_\ell})_{ 1 \leq \ell \leq k-1}.$ % \end{equation} 
Then $(\xi_k Z_{i_k})_{k \geq 1}$ is an i.i.d. sequence.
\end{prop}

\begin{proof}
We proceed by induction on $\ell$ to show that $(\xi_k Z_{i_k})_{1 \leq k \leq \ell}$ is i.i.d. Since $\xi_\ell$ and $i_\ell$ are $\mathcal{F}_{\ell-1}$-measureable, and $i_\ell$ is distinct from $i_1,\ldots,i_{\ell-1}$ we have
	\[ \E ( \xi_\ell Z_{i_\ell} | \mathcal{F}_{\ell-1}) = \xi_\ell \E (Z_{i_\ell} | \mathcal{F}_{\ell-1}) = 0. \]
Since $\xi_\ell Z_{i_\ell}$ is $\pm 1$-valued the proof is complete.
\end{proof}
Recall $\Omega$ the set of all stacks $\omega$ defined by the original process, and  $\Omega',$ the set of all stacks $\omega'$ defined by the modified process. Let $\tau(\omega)$ denote the stopping time of the sequence $\omega$ and similarly $\tau(\omega')$. Furthermore, denote by $(o_\tau,w_\tau)(\omega)$ the \emph{final configuration} after performing $\omega$, and by $(o'_\tau,w'_\tau)(\omega')$ the final configuration of the process after performing $\omega'$.
\begin{lemma}\label{equi12}
There is a measure-preserving map $\phi : \Omega' \to \Omega$ such that with the leftmost convention, 	\[ (\eta'_1(\cdot,\cdot), \eta'_2(\cdot,\cdot)) \stackrel{d}{=} (\eta_1(\cdot,\cdot), \eta_2(\cdot,\cdot)). \]
In particular, $\tau(\phi(\omega')) = \tau(\omega')$,
$(o'_{\tau},w'_{\tau})(\omega') = (o_\tau,w_\tau)(\phi(\omega'))$ and the odometer counts in both the processes are same.

\end{lemma}

\begin{proof}
Given $\omega' \in \Omega'$, let $x_1,x_2,\ldots$ be the resulting leftmost legal firing sequence, and let $u(x)$ be the number of firings performed at $x$ in the modified process.  We set $\phi(\omega') = (X^x_i,Y^x_i)_{x \in \Z, i \in \N}$ where if $i \leq u(x)$ then $X^x_i$ (resp.\ $Y^x_i$) is the direction in which the $i$th oil (resp.\ $i$th water) exited $x$ in the modified process.  

If $i > u(x)$ then we make an arbitrary choice (e.g.\ split the unused portion of each merged stack into even and odd indices).  

It is important to note, that each $x_k$ is measurable with respect to the $\sigma$-algebra $\mathcal{F'}_{k-1}$ generated by the stack variables in $\omega'$ used before time $k$, and that each stack variable in $\omega'$ is used at most once. The stack variables used \emph{at} time $k$ have the form $\xi X$, $\xi Y$ where $X,Y$ are stack variables not yet used and $\xi = 1_{x_k \in 3\Z \cup (3\Z+1)} - 1_{x_k \in 3\Z-1}$ is a random sign that is measurable with respect to $\mathcal{F'}_{k-1}$. Conditional on $\mathcal{F'}_{k-1}$, $\xi X$ and $\xi Y$ are independent uniform $\pm 1$ random variables, and $\phi$ is measure-preserving by Proposition~\ref{p.cards}.
\end{proof}

\old{	
	\begin{align*} (o_{k+1},w_{k+1}) = \begin{cases}
		(o_k - \delta(x_k) + \delta(x_k+\bar{X}^{x_k+1}_{j_k}), \,
		 w_k - \delta(x_k) + \delta(x_k+\bar{Y}^{x_k+1}_{j_k})) & x_k \in 3\Z -1 \\
		(o_k - \delta(x_k) + \delta(x_k+{X}^{x_k}_{i_k}), \,
		 w_k - \delta(x_k) + \delta(x_k+{Y}^{x_k}_{i_k})) & x_k \in 3\Z  \\
		(o_k - \delta(x_k) + \delta(x_k+\bar{X}^{x_k-1}_{j_k}), \,
		 w_k - \delta(x_k) + \delta(x_k+\bar{Y}^{x_k-1}_{j_k})) & x_k \in 3\Z +1
		 \end{cases}		 
		 \end{align*}
}

\subsection{Constant convention}
To avoid cumbersome notation, we will often use the same letter (generally $C$, $C'$, $c$, $c'$, $\e$ and $\delta$)  for a constant whose value may change from line to line.

\section{Preliminary bound on the stopping time}\label{section:apriori}
%!TEX root = oil-and-water.tex

In this section we prove a preliminary bound on  the stopping time $\tau$,  defined in \eqref{stoptimedef}.
The bound is very crude and will be improved in the next section where we prove the upper bound in Theorem \ref{lemmaJ}. However this bound will be used in several places throughout the article.

\begin{lemma}\label{aprioridom}$\tau$ is finite almost surely. Moreover
\begin{enumerate}
\item[i.] For any given $\epsilon>0$,  there exists $c=c(\e)>0$ such that 
$$\P\left(\tau>n^{4+2\epsilon}\right)\le e^{-n^{c}}.$$
\item[ii.]
$\E(\tau)\le16n^4.$
\end{enumerate}
\end{lemma}
This result has an immediate consequence, given by the next Corollary.
\begin{corollary}\label{preliminarysupport} Given $\e>0$ there exists a constant $c$ such that 
$$\mathbb{P}\left(u(x)=0, \ \forall\, x\in \mathbb{Z}:|x|\ge n^{4+\e}\right)\ge 1-e^{-n^c}.$$
\end{corollary}
Before proving Lemma \ref{aprioridom} we recall from subsection \ref{miop} that 
 \begin{equation}\label{Pt}
P_t= \sum_{x\in \mathbb{Z}}\min (\eta_1(x,t),\eta_2(x,t))
\end{equation}
 denotes the total number of co-located oil and water pairs at time $t.$
By definition, the process stops at $\tau$ when $P_\tau=0.$
In order to prove Lemma \ref{aprioridom} we follow the four steps described in the following.
\pagebreak[3]
\begin{itemize}
\item First we define a lazy reflected random walk $R_i$ on $\{0,1,\ldots, n\}$ started and reflected at $n$, and stopped upon hitting $0$. %We call its stopping time 
Define the following stopping time:
\begin{equation}\label{rwhittime}
\tau':=\min_{i\geq 1} \{R_i=0\}.
\end{equation}
\item Consequently, we use $P_j$ to define a series of stopping times $t_i$, and we show that we can bound the tails of the waiting times defined as
\begin{equation}\label{wait11}
\W_i:=t_{i+1}-t_i.
\end{equation}
\item Next, we show that $P_{t_i}$ is stochastically dominated by $R_i$.
\item Finally, we prove Lemma \ref{aprioridom} by combining information about the distributions of $\tau'$ and the waiting times $\W_i$.
\end{itemize}
We now define a lazy random walk. The standard definition is when the simple random walk on $\mathbb{Z}$ does not jump with probability $1/2$ and otherwise jumps uniformly to one of the neighbors.  Throughout the rest of the article many arguments involve random walks with varying degree of laziness. Here 
the lazy reflected random walk $R_i$ is defined as follows. It starts with $R_0=n$. Inductively, if $R_i=n$ then $R_{i+1} =n$ or $n-1$, with probability $3/4$ and $1/4$ respectively. 
Otherwise, if $R_i \in (0,n)$ then it performs a lazy random walk, i.e.,
\[
R_{i+1} =
\begin{cases}
R_i+1 & \textnormal{ w.p.\ }1/4\\
R_i-1 & \textnormal{ w.p.\ }1/4\\
R_i & \textnormal{ w.p.\ }1/2
\end{cases}\, .
\]
The walk terminates at the stopping time $\tau'$, defined in Equation (\ref{rwhittime}).

To perform the second and third steps, we analyze $P_t$ and show that it is a super-martingale. This was described heuristically in subsection \ref{miop} and here we make this concept formal, as this fact is the key to many of the subsequent arguments.  

If at time $t$ a site $x$ emits a pair, then $P_{t+1}$ is either $P_t-1, P_t,$ or $P_t+1$. Moreover the distribution of $P_{t+1}-P_t$ depends on the state of the neighbors of $x$ at time $t$ in a fairly simple way.

Define $\ft$ to be the filtration given by the first $t$ firings. 
Recall from Definition \ref{particlenotation} that $\eta_1(x,t)$ and $\eta_2(x,t)$ are the number of oil and water particles, respectively, at $x$ at time $t$.
\begin{definition}\label{deftype}Define for any $x \in \mathbb{Z}$ and non negative integer time $t$ including infinity  
\begin{equation}\label{eq:g_t}
g_t(x):=\eta_1(x,t)-\eta_2(x,t).
\end{equation}
We say that at time $t$ a site $x\in \mathbb{Z}$ has \textit{type} $oil$, $ 0 $ or $water$ depending on the type of the majority of particles at the site, i.e., whether $g_t(x)$ is positive, zero or negative respectively. 
\end{definition}

We look at how $P_t$ changes conditional on the filtration at time $t$. Formally, we look at
\begin{equation}\label{conditionalincrement}
Z_t:=P_{t+1}-P_{t} .%\eta_1(\cdot,t),\eta_2(\cdot,t).
\end{equation}

At this point we observe that $Z_t$ can have four possible distributions conditioned on $\ft$. Namely, define four random variables $\xi_1,\xi_2,\xi_3,\xi_4 $ as follows: 
\[
\begin{array}{ll}
\xi_1:= \left \{
\begin{array}{cl}
0 & \textnormal{w.p. } 1/2\\
1 & \textnormal{w.p. } 1/4\\
-1 & \textnormal{w.p. } 1/4
\end{array}
\right.
\qquad
& \xi_2:= \left \{
\begin{array}{cl}
0 & \textnormal{w.p. } 3/4\\
-1 & \textnormal{w.p. } 1/4
\end{array}
\right. \\[5ex]
%\]
%\[
\xi_3:= \left \{
\begin{array}{cl}
0 & \textnormal{w.p. } 1/2\\
-1 & \textnormal{w.p. } 1/2
\end{array}
\right.
\qquad
& \xi_4:=0 \ \textnormal{ w.p. } 1.
\end{array}
\]
\begin{remark}\label{distributionremark}
Every time that a vertex $x$ fires, we divide the possible \textit{types} of  the neighbors of $x$ into four groups which determine the law of $Z_t$. %Hence:
\[
Z_{t} \stackrel{\text{law}}{=}
\left \{
\begin{array}{cl}
\xi_1 & \textnormal{if the neighbors have different } \textit{types} \textnormal{ but neither is }0 \\
\xi_2 & \textnormal{if exactly one of the neighbors has  } \textit{type } 0\\
\xi_3 & \textnormal{if both neighbors have  } \textit{type }0\\
\xi_4 & \textnormal{if the neighbors have the same nonzero  } \textit{type}.
\end{array}
\right.
\]
\end{remark}

\begin{definition}
A {\bf firing rule} is an inductive way to determine a legal firing sequence. More formally it is a function that is defined for any $t$ and any atom  $A \subset \ft$ such that $P_t>0$. The firing rule outputs an integer $z$ such that there is at least one pair at $z$ at time $t$ in $A$.
\end{definition}
The stacks combined with a firing rule determine the evolution of the process.
Given a firing rule, we inductively define the set of stopping times $t_i$. Set $t_1=1$. 

Then, inductively, given $t_1,\dots,t_i$ set $t_{i+1}$ to be the first time $j$ after $t_i$ such that $Z_j$ does not have distribution $\xi_4$. Let $t_L$ be the last stopping defined. 
Thus 
\begin{equation}\label{noofjumps}
L:=\#\{i\leq \tau \ : \ Z_i\stackrel{\text{law}}{\neq} \xi_4\}.
\end{equation}
Note that with this definition of the stopping times we have
$$P_{t_{i+1}}=P_{t_i+1}.$$ 
Thus the distribution of $P_{t_{i+1}}-P_{t_i}$ is either $\xi_1,$ $\xi_2$ or $\xi_3$.
Finally, define the waiting times 
$$\W_i=t_{i+1}-t_i.$$
With these definitions we are ready to complete the second and third steps of our outline.
\begin{lemma} \label{firingrule}
For any firing rule the sequence 
$R_i$ stochastically dominates $P_{t_i}$. This in particular implies that  $\tau'$ stochastically dominates $L.$
\end{lemma}

\begin{proof}
The proof is by induction, and the starting configuration is given by $R_0=P_0=n$.
If $R_i=n$ then the distribution of $R_{i+1}-R_i$ is $\xi_3$, by definition of the reflected random walk.

Inductively, if $P_{t_i}\geq R_i=n$ then $P_{t_i}$ must also be $n$. This means that all sites have type $0$ and the distribution of $P_{t_{i+1}}-P_{t_i}$ is $\xi_3$. Thus we can couple $R_{i+1}$ and $P_{t_{i+1}}$ such that 
$R_{i+1}\geq P_{t_{i+1}}.$

On the other hand, if $0<R_i<n$ then then the distribution of $R_{i+1}-R_i$ is $\xi_1$ while the distribution of $P_{t_{i+1}}-P_{t_i}$ is either $\xi_1,$ $\xi_2$ or $\xi_3$. As all of $\xi_1,$ $\xi_2$ and $\xi_3$ are stochastically dominated by $\xi_1$ we can couple $R_{i+1}$ and $P_{t_{i+1}}$ such that 
$R_{i+1}\geq P_{t_{i+1}}.$
\end{proof}

\begin{lemma} \label{hayden}

For any $\epsilon>0$ there exists $\delta$ such that 
$$\P(L>n^{2+\epsilon})\leq\P(\tau'>n^{2+\epsilon})<e^{-n^{\delta}}.$$
where $\tau'$ and $L$ are defined in \eqref{rwhittime} and \eqref{noofjumps} respectively. 
\end{lemma}
\begin{proof}
The first inequality follows from the fact that $\tau'$ stochastically dominates $L.$  . The second inequality is a standard fact about lazy random walk  (stated later in Lemma \ref{mean1}).
\end{proof}
\begin{lemma} \label{waitingtimes}
Let $E_i$ be an i.i.d.\  sequence of random variables whose distribution is the same as the time taken by  simple symmetric random walk started from the origin to hit $\pm 2n$.
There exists a firing rule such that $\W_i$ is stochastically dominated by $E_i$.
\end{lemma}

\begin{proof}
At any stopping time $t_i$, our firing rule is to pick the location of the left most pair. For times $j \neq t_i$ for any $i$, we define the firing rule inductively. If $t_i< j<t_{i+1}$ then at time $j$ we choose to fire the pair at the location of the oil particle that just moved at step $j-1$. Thus we have an oil particle performing simple random walk until the distribution of $Z_j$ is not $\xi_4$. We now show that this firing rule allows us to control the waiting times.

Let $z$ be the location of the leftmost pair at time $t_i$ and let $I=(a,b)$ be the interval of integers containing $z$ where every location has at least one particle at time $t_i$. As there are at most $2n$ particles we have that $I=(a,b) \subset (z-2n,z+2n)$.

If there are firings only in the interior of $I$ then there are no particles at $a$ or $b$. If the oil particle reaches the boundary of $I$  at  time $k$ then the distribution of $Z_k$ is not $\xi_4$, as one of the neighbors of the pair to be fired (either $a$ or $b$) has no particles and is of \textit{type} $0$. Thus $t_{i+1}\leq k$. This tells us that the distribution of the waiting time $\W_i=t_{i+1}-t_i$ is bounded by the time taken by  simple random walk started at the origin to leave the interval  $(-2n,2n)$. Since this is true for each $i$ independently, the lemma is proven.
\end{proof}

We are now ready to prove Lemma \ref{aprioridom}. 

\begin{pfoflem}{\ref{aprioridom}}
We first notice that by Lemma \ref{l.abelian},  $\tau$ defined in \eqref{stoptimedef} is independent of  the firing rule. For the purposes of the proof we will fix our firing rule to be the one  mentioned in Lemma \ref{waitingtimes}. Now by Lemmas \ref{firingrule} and \ref{waitingtimes}  $$\sum_{i=1}^{\tau'}{E_i} \mbox{ stochastically dominates } \tau.$$    
Thus $$\pr(\tau >n^{4+\epsilon}) \le \pr (\sum_{i=1}^{\tau'}{E_i} \ge n^{4+\epsilon} )  $$

If $\sum_{i=1}^{\tau'}{E_i} \ge n^{4+\epsilon} $ then either
$$\tau'>.1 n^{2+\epsilon/2}\qquad \text{ or }
\qquad \sum_{i=1}^{.1 n^{2+\epsilon/2}} E_i > n^{4+\epsilon/2}.$$
Therefore 
\begin{equation*}
\label{tail2}
\P(\tau >n^{4+\epsilon} )\leq \P(\tau'>.1 n^{2+\epsilon/2})+\P\left(\sum_{i=1}^{.1 n^{2+\epsilon/2}} E_i > n^{4+\epsilon}\right)
\end{equation*}
The first term is bounded by Lemma \ref{hayden}. The second is bounded using standard bounds on the distribution of the time for simple random walk started at the origin to leave a fixed interval $(-k,k)$. Both of these probabilities are bounded by $e^{-n^{\delta}}$ for some 
$\delta=\delta(\epsilon)>0$.

Furthermore, by Lemma \ref{waitingtimes} we have that the $E_i$'s are all i.i.d., hence we can apply Wald's identity (see e.g., \cite[Section 3.1]{durrett}) and get that 
$$\E(\tau)\le\E(\tau')\E(E_1) \le 16n^4,$$ 
which concludes the proof.
\end{pfoflem}

\section{Proof of Theorem \ref{lemmaJ} (Upper bound)} \label{upperbound}
%!TEX root = oil-and-water.tex

\subsection{Notation} \label{couplings}

In order to proceed, we need to introduce some further notation. Let
\begin{equation*}
\Height=\max \{u(x): \ x \in \Z\}.
\end{equation*}
In this section we will bound from below the probability of the event
\begin{equation} \label{mainevent}
\m =\{\Height \leq C n^{4/3}\}.
\end{equation}
for some large constant $C$.
%We start by describing the notation that we will need for this part of the proof which is slightly different than what we have been using in the previous sections.
To this purpose, we define
\[
\returns(x):=\#\left \{t \leq \tau \ : \ \eta_1(x,t)=\eta_2(x,t), \ x_t\in \{x-1,x+1\}\right \}
\] 
$\returns(x)$ is the number of times such that
\begin{enumerate}
\item $x$ has type $0$, (the same number, possibly $0$ of oil and water particles), and
\item a pair of particles is emitted from $x-1$ or $x+1$.
\end{enumerate}
Then define
\begin{equation} \label{advantage}
\displaystyle{\returns =\sum_{x \in \Z}\returns(x)}.
\end{equation}
Recall the definition of $N_i(t)$ from Equation (\ref{eq:def_N_i}) and notice that
\begin{equation}\label{eq:ret_N}
\returns=N_2(\tau)+2N_3(\tau).
\end{equation}
Our goal is to show that $\returns$ is large. The advantage of the decomposition in \eqref{advantage} is that it will allow us to show the relationship between $\returns$ and the odometer function.

Finally, recall that $\mathcal{G}_t$ is the $\sigma$-algebra generated by the movement of the first $t$ pairs that are emitted and notice that, from the definition of $P_t$  \eqref{Pt} we have that
\[
P_t=
\sum_{x\in \Z}  \min\left [\eta_1(x,t),\eta_2(x,t) \right]
\]
is the number of pairs at the time of the $t^{th}$ emission.
We recall that $P_t$ is a supermartingale with respect to $\mathcal{G}_t$.

\subsection{Outline}
In order to show that it is unlikely that $\height$ is much larger than $n^{4/3}$, we rely on three main ideas.
\begin{itemize}
\item[(i)] {\bf The odometer is fairly regular.} 
Typically 
$$|u(x)-u(x+1)| \leq 2n +u(x)^{1/2}.$$ 
This regularity 
implies that under the assumption that $u(x_0)$ is much bigger than $n^{4/3}$ %then by \eqref{thatiswhatshesaid} 
then it is likely that $u(x)$ is much bigger than $n^{4/3}$ for all $x$ such that $|x-x_0|<n^{1/3}$.
\item[(ii)]
%The second key fact is showing that  
{\bf The variable $\returns$ is comparable with the odometer function.} Fix $x \in \Z$ and consider $\eta_1(x,k)-\eta_2(x,k)$ as a function of $k$.
This function performs a lazy random walk that takes about $$\frac12\big(u(x-1)+u(x+1)\big)\approx u(x)$$ steps. Thus we expect $\returns(x)$ to be on the order of $u(x)^{1/2}.$ Summing up over all $x$
we expect
\begin{equation} \label{thatiswhatshesaid}
\returns =\sum_{x \in \Z} \returns(x) \approx \sum_{x \in \Z} u(x)^{1/2}.
\end{equation}
Combined with the previous paragraph this implies that if $\height$ is much larger than $n^{4/3}$ then it is likely that 
$\returns $ is much larger than $n.$
\item[(iii)]
%The final key fact is that 
{\bf The process $P_t$ is a supermartingale.} 
$P_t$ starts at $n$ and stops when it hits 0. The sum of the negative drifts until the process terminates is, by Equation (\ref{eq:ret_N}), 
$$-\frac14(N_2(\volume)+2N_3(\volume)) =-\frac14 \returns.$$ 
Hence, its stopping time is typically around $4\returns$. We use the Azuma-Hoeffding inequality to show that the probability of the event ``$\returns$ is much larger than $4n$'' is decaying very rapidly. By the previous paragraphs we also will get that the probability that $\height$ is much larger that $n^{4/3}$ is decaying very rapidly.
\end{itemize}
The main difficulty in implementing this outline comes in the second step as the odometer function and $\returns$ are correlated in a complicated way.

\subsection{The bad events} \label{be}
To make our outline formal we now define our set of bad events. 
The first three of these deal with the odometer function. But first we use the odometer function to define 
\[
\base=\{x:\ u(x)>0\} \subset \Z.
\]

\begin{itemize}
\item[(0)] $\badzero=\big\{\base \not \subset [-n^5,n^5]\big\}$
\item[(i)] $\badone=\{\Height \geq n^{4.05}\}$. 
\item[(ii)] {\bf Gradient of the Odometer.} In words, $\badtwo$ is the event where the \emph{gradient}  of the odometer function, $u(x)-u(x+1)$, is too large or has the wrong sign. 

Let $$m(x)=\begin{cases}
\min\left\{n,\min_{0 \leq y\leq x}u(y)\right\} & \text{ for $x >0$ and }\\
\min\left\{n,\min_{x \leq y\leq 0}u(y)\right\} &\text{ for $x \leq 0$}.
\end{cases}
%\right.
$$
 We define:

\begin{eqnarray*}
\badtwo :=
 \left\{\exists \,\, x \in [-n^5,n^5]:\,\, u(x)\ge n^{.5}\mbox{ and }|u(x+1)-u(x)|>2m(x) +u(x)^{.51}\right\}  \\
\bigcup \left\{\exists \,\, x \in [0,n^5]:\,\,u(x)\ge n^{.5}\mbox{ and } (u(x+1)-u(x))>u(x)^{.51}\right\}  \text{} \hspace{.15in}\\ 
 \bigcup \left\{\exists x \in [-n^5,0]:\,\, u(x)\ge n^{.5}\mbox{ and }\,\,(u(x-1)- u(x))>u(x)^{.51}\right\} .
\end{eqnarray*}

\begin{remark} \label{rising}
While the definition of $\badtwo$ is quite technical we now give one consequence of it that is representative of how we will use it.
Let  $\badzero^{\comp} \cap \badtwo^{\comp}$ occur and consider the set of $x$  such that 
$n<u(x)<n^{1.96}$ (or equivalently $1<\log_n(u(x))<1.96$). Then
\begin{equation} \label{thriller}|u(x)-u(x+1)|<2m(x)+u(x)^{.51} \leq 2n+(n^{1.96})^{.51} \leq 3n.\end{equation}
For any $x$ and $y$ in a connected component of this set then \eqref{thriller} implies 
$$\text{ if
$|\log_n(u(x))-\log_n(u(y))|\geq .05$} \qquad
\text{then $|x-y|>\frac16 n^{.05}.$}$$
This is because our conditions on $x$ and $y$ imply $|u(x)-u(y)|\geq n^{1.05}-n>.5n^{1.05}$. So between $x$ and $y$ the odometer changes by at least $.5n^{1.05}$ in increments of at most $3n$.
\end{remark}
Our proof makes heavy use of this and similar estimates  that follow from 
$\badzero^{\comp} \cap \badtwo^{\comp}$.  Much of the complication of our proof comes from the fact that we need to use different estimates depending on whether $m(x)=n$ or $m(x)<n$ and whether  $\log_n(u(x))$ is greater than 1.96, between 1 and 1.96 or less than 1. These estimates are used in Lemmas \ref{bruce} and \ref{springsteen} which are used to prove Lemma \ref{existence}. The conclusion of Lemma \ref{existence} is useful in showing that $\returns$ is large because of our next bad event.
\item[(iii)] {\bf Regularity of Returns}
$$\badthree =\bigg\{\exists i,j,k:\ |i|,|j|\leq n^5,\ k \in [n^{.5},n^5],\ j-i>.1n^{.01} \text{ and }
u(x)\geq k\  \forall x \in [i,j]$$ \vspace{-.2in}
$$\hspace{3in}\text{and} \sum_{x=i+1}^{j-1}\returns(x)<.01(j-i)\sqrt{k}\bigg\}$$
\end{itemize}

Finally define the event

\begin{equation} \label{e.theimportantevent}
\star =\big\{\returns \geq 20\left(n + \volume^{.51}\right)\big\}.
\end{equation}

To complete our outline we show that 
\begin{itemize}
 \item $ \m^{\comp} \subset \star \cup \bigcup_{i=0}^3 \badi$ and 
 \item  $\P(\star)$ and the  $\P(\badi)$ are small.% is small for all $i$
 %\item $\P(\star)$ is small.
 \end{itemize}
% $$\m \cap \badzero^{\comp} \cap \badone^{\comp} \cap \badtwo^{\comp} $$
\subsection{$P_t$ is a supermartingale.}
Recall from \eqref{spl} that
\[
Z_k:=P_{k+1}-P_k.
\] 
%As we have discussed in Section \ref{section:apriori} the sequence $\{Z_k\}$ is not independent. 
Recall from Remark \ref{distributionremark} that
conditional on the past, each $Z_k$ has one of four possible distributions $\xi_1,\dots ,\xi_4$.

\begin{lemma} \label{lemmaAT}
$P_t$ is a supermartingale and 
$$\sum_{t=1}^{\infty}\E\left(Z_{t+1} | \mathcal F_{t}\right)=-\frac14 \returns.$$
\end{lemma}
\begin{proof}
This follows from the discussion in Section \ref{miop}. The summands on the left hand side are all $0$, $-\frac14$ or $-\frac12$. By \eqref{eq:ret_N}  and \eqref{eq:def_N_i} $\returns$ is the sum of the number of $-\frac14$ terms in the sum plus twice the number of $-\frac12$ terms. 
\end{proof}
\begin{lemma} \label{lemmaA} There exist $C,C'$ such that
$\P(\star)<Ce^{-n^{C'}}.$ 
\end{lemma}
\begin{proof}
As $\returns \leq \tau$ we have that $$\star=\bigcup_{r \geq 20n}\star \cap \{\tau=r\}.$$ 
Note that by  Lemma \ref{lemmaAT} and because $P_\tau=0$
$$\star \cap \{\tau=r\} \subset \left\{ \sum_{i=1}^r Z_{t}-\E\left(Z_{t} | \mathcal F_{t-1}\right)\geq 4n+5r^{.51}\right\}.$$
Also note that $Z_{t}-\E\left(Z_{t} | \mathcal F_{t-1}\right)$ are the increments of a martingale and are bounded by 1. Thus we can apply the Azuma-Hoeffding inequality to get the bound 
$\P(\star \cap \{\tau=r\})\leq e^{-r^{.02}}. $
The lemma follows from the union bound.
\end{proof}

\subsection{Regularity of the odometer}
Now define the new event 
$$\regular=\m^{\comp} \cap\badzero^{\comp} \cap  \badone^{\comp} \cap \badtwo^{\comp},$$ where all the events are defined in Subsection \ref{be}. Remember the event $\m$ is that the odometer at the origin is less than a large constant times $n^{4/3}$ and the events $\badzero^{\comp}$, $\badone^{\comp}$ and $\badtwo^{\comp}$ are regularity conditions on the odometer.

\begin{lemma} \label{lemmaC}
\begin{equation}\label{felix1}
\regular \ \cap \  \badthree^{\comp} \subset \star
\end{equation} where $\star$ was defined in \eqref{e.theimportantevent}. Thus
\begin{equation}
\pr(\m^{\comp}) \leq \P(\badzero)+\pr(\badone) + \pr(\badtwo ) +\pr(\badthree) +\pr(\star).
\label{felix}
\end{equation}
\end{lemma}
Assuming \eqref{felix1}, \eqref{felix} follows by taking the union bound.

We split the proof of Lemma \ref{lemmaC} into several lemmas. In particular, it will suffice to show that
\begin{eqnarray}\label{eq:returns}
\frac12 \returns & > & 20n \\
\label{eq:returns_volume}
\frac12 \returns & > & 20\volume^{.51}.
\end{eqnarray}

\begin{lemma}
\label{garden}
If the event
$\regular$
occurs, then
\begin{equation}\label{kingfelix}
\volume \geq n^{5/3}
\end{equation} 
and if $\regular \cap \badthree^{\comp}$ occurs, then
$$\frac12 \returns >20n. $$
\end{lemma}
\vspace{.1in}
Recall that on the event 
$\regular$
the height is much bigger than $n^{4/3}$ but less than $n^{4.05}$, the odometer is supported inside 
$[-n^5,n^5]$, and gradient of the odometer is such that $\badtwo^{\comp}$ occurs.  
\begin{proof}
As usual, we choose to give a direct proof with explicit constants (which might be far from optimal), for sake of exposition.

If $\badzero^{\comp}$ occurs there exists $$x_{\Height} \in [-n^5,n^5],\mbox{ such that } u(x_{\height})=\height.$$
The  events $\m^{\comp}$ and $\badtwo^{\comp}$ imply that 
\begin{equation}\label{15now}
u(x)>700000n^{4/3}>(800n^{2/3})^2
\end{equation}
 for all $x $ in the interval
\[
X_0:=[x_{\Height}-100n^{1/3},x_{\Height}+100n^{1/3}] \cap [-n^5,n^5]\cap \base.
\]
This follows from the event $\badzero^{\comp}\cap\badtwo^{\comp}$ by the same argument as in  Remark \ref{rising}.  
Then $|X_0|\geq 100n^{1/3}$
because by $\badzero^{\comp}$ $$x_{\Height} \in [-n^5,n^5]$$ so at least one of 
$$[x_{\Height},x_{\Height}+100n^{1/3}]\mbox{ or } [x_{\Height},x_{\Height}+100n^{1/3}]$$ 
is entirely in $[-n^5,n^5]$. 

Thus looking at the volume of the odometer in $X_0$ we have 

\begin{equation}\label{kingfelix2}
\volume \geq\sum_{x \in X_0}u(x) \geq 700000n^{4/3}|X_0| \geq n^{5/3}
\end{equation} 
and
\begin{equation}
\returns 
\geq \sum_{x \in X_0} \returns(x) 
    \geq (100n^{1/3})(.01)(800n^{2/3}) \geq 40n.
\end{equation}
The next to last inequality is by \eqref{15now}, the size of $X_0$ and the definition of $\badthree^{\comp}$. Thus we have obtained inequality (\ref{eq:returns}).
\end{proof}

\subsection{Partitioning $\base$.} \label{partitioning}
Inequality (\ref{eq:returns_volume}) is more involved to verify, and in order to do it we proceed as follows. 
We partition $\base$ up into smaller intervals:
\[
\base =\bigcup_{i=-K'}^K\base_{i}
\]
where, for each $i$, we set each $\base_i=[a_i,b_i]$ for some $a_i$ and $b_i$ which we describe in the following. 
We inductively define the $\base_i$ in a way so that on the event 
$\regular$
and for most $i$ we have
\begin{equation}\label{eq:either}
b_i-a_i>.1n^{.01}.
\end{equation}
This will involve a series of estimates like in Remark \ref{rising}.

Define $h_j=n^{4.05-.05j}$ for $j=0$ to $j=68$ and $h_{69}=0$. Let 
$$\h=\{h_j:\ j=1\dots 68 \}.$$
Let $j'$ be such that $h_{j'}$ is the closest value in $\h$ to $u(0)$. 
Let $$b_0= \inf \left\{x:\ x\geq 0 \text{ and } u(x+1) \not \in (h_{j'+1},h_{j'-1})\right\}.$$
We say that $\base_0$ \emph{starts} at height $h_{j'}$ and \emph{ends} at height $h_{j''} \in \h$ where $j''$ is defined so that $h_{j''}$ is the closest element of $\h$ to $u(b_0+1).$ 

Now we inductively define $\base_{i+1}$.
Suppose we have defined $\base_i=[a_i,b_i]$ which ends at height $h_k$. We will inductively define $\base_{i+1}=[a_{i+1},b_{i+1}].$ 
We let $a_{i+1}=1+b_i$ and say $\base_{i+1}$ \emph{starts} at height $h_k \in \h$. 
Then we define
$$b_{i+1}= \inf \left\{x:\ x \geq a_i \text{ and } u(x+1) \not \in (h_{k+1},h_{k-1})\right\}.$$
We say the block $\base_{i+1}$ \emph{ends} at height $h_{k'} \in \h$ where 
$h_{k'}$ is the closest element of $\h$ to $u(b_i+1)$ and $h_{k'}$ is the closest element of $\h$ to $u(b_i+1).$

For the case of negative indices, the procedure is totally analogous.
%We do a similar procedure to define $\base_i$ for negative $i$.
Finally, for all $j \in \{0,\dots, 68\}$ define 
\begin{equation}\label{eq:def_B_j}
B_j:= \textnormal{the union of all }\base_i\textnormal{ that start at $h_j$ 
and have }|\base_i|>.1n^{.01}. 
\end{equation}

\begin{lemma} \label{bruce}
If a block $\base_i$ with $i\geq 0$ starts at height $h_{j}$ and ends at height $h_{j'}$ then $j'\neq j$. 
Conditional on $\regular$ and $j'<j<68$ then
$$b_i-a_i>.1n^{.01}.$$
\end{lemma}

\begin{proof}
The first statement is true because $$u(1+b_i)\not\in (h_{k+1},h_{k-1})$$ so the closest element of $\h$ to 
$u(1+b_i)$ is not $h_k$.
Consider an interval $\base_i$ with $i\geq0$ where $\base_{i}$ starts at
height $h_j$ and ends at height $h_{j'}$ with $j'<j$.
Over the course of such an interval the odometer increased by at least $.49h_{j-1}$, going from less than $.5(h_j+h_{j-1})$ to at least $h_{j-1}.$ Since $j<68$ we have $$u(x)>h_{j+1}\geq h_{68}=n^{.65}>n^{.5}$$  for all $x \in \base_i$. Thus the event $\regular$ implies 
 $$u(x+1)-u(x)<u(x)^{.51}<(h_{j-1})^{.51}.$$
Thus by the choice of the $h_i$ these intervals must have width at least 
$$b_i-a_i\geq \frac{.49h_{j-1}}{(h_{j-1})^{.51}}\geq .49(h_{j-1})^{.49} \geq .49n^{.3}>.1n^{.01}.$$
The next to last inequality follows because $j<68$ so $h_{j-1}\geq n^{.7}.$
\end{proof}

\begin{lemma} \label{springsteen}
If $\regular$ occurs and a block $\base_i$ with $i>0$ starts at height $h_{j}$ and ends at height $h_{j'}$ with $j'>j$ then
\begin{enumerate}
\item $b_i-a_i>.1n^{.01}$ or
\item $68 \geq j\geq 60$ and for no $k$, $0\leq k<i$ the block $\base_k$ starts at $h_j$ and ends at a height $h_{k'}$ with $k'>j$.
\end{enumerate}
\end{lemma}

\begin{proof}
Consider an interval $\base_i$ with $i\geq0$ where $\base_{i}$ starts at
height $h_j$ and ends at height $h_{j'}$ with $j'>j$.
Over the course of such an interval the odometer decreased by at least $.49h_{j}$, going from at least $.5(h_j+h_{j+1})$ to at most $h_{j+1}.$ 

First we consider the case that $j<60$. We have $$u(x)>h_{j+1}\geq h_{60}=n^{1.05}$$  for all $x \in \base_i$. Thus the event $\regular$ implies 
 $$u(x)-u(x+1)<2m(x)+u(x)^{.51}=2n+u(x)^{.51}<3(h_{j})^{.96}.$$
Thus by the choice of the $h_i$ these intervals must have width at least 
$$b_i-a_i\geq \frac{.49h_{j}}{3(h_{j})^{.96}}\geq .15(h_{j})^{.04} >.1n^{.01}.$$

Next we consider the case that $\base_i$ ends at $h_j$ with $60 \leq j \leq 68.$ Suppose there exists $k<i$ such that $\base_k$ starts at $h_j$ and ends at a height $h_{k'}$ with $k'>j$. This implies that
$j<68$ and 
$$m(x)\leq h_{j+1} \leq (h_j)n^{-.05}$$ for all $x \in \base_i$.
Thus the event $\regular$ and $j\leq 67$ implies $u(x)>n^{.65}$ and 
 $$u(x)-u(x+1)<2m(x)+u(x)^{.51}<3h_jn^{-.05}.$$
Thus these intervals must have width at least 
$$b_i-a_i\geq \frac{.49h_{j}}{3(h_{j})n^{-.05}}\geq .15n^{.05} >.1n^{.01}.$$
\end{proof}

\subsection{Consequences of a regular gradient.}
Recall the definition of $B_i$ from \eqref{eq:def_B_j}.
\begin{lemma} \label{existence}
Conditional on $\regular$
there exists $i \in \{1,\dots,68\}$ such that 
$$\sum_{x \in B_i}u(x) \geq .01 \volume.$$
\end{lemma}

\begin{proof}
Define
$$B^*=\bigcup_{i:|\base_i|\leq .1n^{.01}} \base_i.$$
We first show that
\begin{equation} \label{hatchet}\sum_{x \in B^*} u(x) \leq 20n^{1.2}.
\end{equation}
The lemma will follow easily from \eqref{kingfelix}, \eqref{hatchet}  and the pigeonhole principle.

Let $I=\{i:|\base_i|\leq .1n^{.01}\}.$
First we show that
$|I|\leq 20.$ By Lemmas \ref{bruce} and \ref{springsteen} for every $i \in I$ there exists $j$, $60 \leq j \leq 68$ such that 
$\base_i$ starts at height $h_j$. From this we draw two conclusions.
First by the definitions of $B^*$ and the $\base_i$ we have $u(x)<n^{1.15}$ on $B^*$. Also for each such $j$ Lemma \ref{springsteen} implies there exist at most two $i \in I$ with $\base_i$ starting at height $h_j$, at most one with $i\geq 0$ and at most one with $i \leq 0$.  
These two facts combine to  establish \eqref{hatchet} which completes the proof.
\end{proof}

\begin{remark}
If we perform the analysis in the previous lemma to non-empty intervals $\base_i$ that start at $h_j \geq n^{1.4}$ we get that $|\base_i|\geq .1n^{.4}$. This implies 
$$\tau \geq .1n^{.4}h_{j+1}=.1n^{.3}h_{j-1}.$$
\end{remark}
%\textcolor{red}{again this is not completely precise because the value at which it ends is not  $h_{j+1}$ but possibly smaller.}
\begin{lemma} \label{ratio}
If $B_j \neq \emptyset$ and $\regular$ occurs, then
$$\frac{\volume^{.49}}{(h_{j-1})^{.5}}\geq n^{.1}.$$
\end{lemma}
\begin{proof}
If $h_j \leq n^{1.35}$ then the result follows from the first part of Lemma \ref{garden}.  If $h_j \geq n^{1.4}$ then it follows from the previous remark.
\end{proof}

\subsection{Consequences of Regular Returns.}
Recall the sequence $h_i$ defined in Subsection \ref{partitioning}. 
\begin{lemma} \label{library}
For every $j \in \{1,\dots,68\}$ the set $B_j$ satisfies
$$ \sum_{x \in B_j}u(x)\leq |B_j|h_{j-1}.$$
Conditional on $\regular \cap \badthree^{\comp}$  for every $j \in \{0,\dots,67\}$ the set $B_j$ satisfies
$$ \sum_{x \in B_j}\returns(x)\geq .01| B_j| (h_{j+1})^{1/2}.$$
\end{lemma}

\begin{proof}
From the choice of the intervals $\base_j$ and $B_i$ we have that
$$h_{i+1}\leq u(x) \leq h_{i-1},$$
for all $x \in B_i$. Also, by definition, $B_i$ consists of a union of intervals of width at least $.1n^{.01}$ (cf.\ Equation (\ref{eq:def_B_j})).
Therefore the second statement follows from the definition of $\badthree^{\comp}$.
\end{proof}

\begin{lemma} \label{volunteer}
For all sufficiently large $n$
conditional on $\regular \cap \badthree^{\comp}$
we have
$$\frac12 \returns >20\volume^{.51}. $$
\end{lemma}
 \begin{proof}
As $\regular$ occurs, by Lemma \ref{existence} we obtain that there exists a $j$ such that 
$$\sum_{x \in B_j}u(x) \geq .01 \volume.$$
First, consider the case that $j \in \{1,\dots,67\}$. In this case, from Lemma \ref{library} it follows that 
\begin{equation}\label{eq:ineq_B_j}
|B_j|h_{j-1} \geq \sum_{x \in B_j}u(x)>.01\volume.
\end{equation}
Conditional on the event $\badthree^{\comp}$, the set $B_j$ satisfies $\sum_{x \in B_j}\returns(x)\geq .01| B_j| (h_{j+1})^{1/2}$, which implies
\begin{eqnarray*}
\sum_{x \in B_j}\returns(x)&\geq &.01| B_j| (h_{j+1})^{1/2}\\ 
&\stackrel{(\ref{eq:ineq_B_j})}{\geq} &.01\left(\frac{.01\tau}{h_{j-1}}\right)(h_{j+1})^{1/2}\\
&= & .0001\left( \frac{h_{j+1}}{h_{j-1}}\right)^{1/2} \frac{\tau^{.49}}{(h_{j-1})^{.5}}\ \ 
\volume^{.51}\\
&\geq & .0001 n^{-.05}n^{.1} \volume^{.51}\\
&\geq & 40 \volume^{.51}.
\end{eqnarray*}
The next to last line follows from Lemma \ref{ratio}, together with the definition of $h_{j+1}$ and $h_{j-1}$, whereas the last inequality holds whenever $n$ is sufficiently large.

Now consider $j=68$. In this case, by Lemma \ref{garden} we have $\volume>n^{5/3}$ so $\tau^{.49}>n^{.8}$ and
$$\sum_{x \in B_{68}}\returns(x) \geq |B_{68}|\geq \frac{.01\volume}{n^{.7}}
= \frac{.01\volume^{.51}\volume^{.49}}{n^{.7}} \geq 40\tau^{.51},$$
which ends the proof.
\end{proof}

\begin{pfoflem}{\ref{lemmaC}}
The proof is an easy consequence of Lemmas \ref{garden} and \ref{volunteer}.
\end{pfoflem}

\subsection{The probability of the bad events} \label{probs}
Finally, we can use our results to bound the probability of $\m$.
In Lemma \ref{lemmaC} we showed that the event $\m^{\comp}$ is contained in a certain union of the bad events and then used union bound to bound its probability. We now bound the probability of those bad events.

\begin{lemma} \label{lemmaF}
There exist positive constants $C, C'$ and $\gamma$ such that
$$ \pr(\badtwo)
< Ce^{-C'n^{\gamma}}.$$
\end{lemma}
We first introduce a definition.
\begin{definition}\label{notationfluc1}
Let $\Delta^x(k)$ denote the quantity such that, after $k$ pairs have been emitted from $x$, there are $k+\Delta^x(k)$ particles that moved to the right (i.e.\ to $x+1$) and $k-\Delta^x(k)$ particles that have moved to the left (i.e.\ to $x-1$). Notice that this is just a function of the stack of variables $X^{x}_i,Y^{x}_i$ at the site $x$.
\end{definition}

\begin{proof}[Proof of Lemma \ref{lemmaF}]
Without loss of generality, we assume that $x \geq 0$. 
Recall from Subsection \ref{be} (ii) that
\[
m(x)=\min\left \{n,\min_{0 \leq y\leq x}u(y)\right \}.
\]
Suppose that at some time $t$ exactly $k$ pairs have been emitted from $x$, and exactly $k'$ pairs have been emitted from $x+1$. Then, the number of particles to the right of $x$ is
\[
k+\Delta^x(k)-(k'-\Delta^{x+1}(k'))
\]
which is between 0 and $2m(x)$. This holds for all times $t$, in particular it holds for the time when the process stops. Now set $t$ to be such that $k=u(x)$ and $k'=u(x+1).$
Let $\Delta^x=\Delta^{x}(u(x))$ and $\Delta^{x+1}=\Delta^{x+1}(u(x+1)).$ Then
$$0 \leq u(x)+\Delta^x-(u(x+1)-\Delta^{x+1}) \leq 2m(x). $$
Rearranging, we get
$$
-\Delta^x-\Delta^{x+1} \leq u(x)-u(x+1)
    \leq 2m(x) -\Delta^x-\Delta^{x+1}. 
$$
Then the event $\badtwo$ implies that 
$$\exists x \in [-n^5,n^5] \text{ and } k,\ n^{.5}\leq k \leq n^{4.1}:\ |\Delta^x(k)|>.5k^{.51}.$$
The result follows from standard concentration results of random walks (cf. \ref{section:ConcentrationEstimates}) and union bounding over all possible values of $x$ and $k$. 
Thus the result holds for some appropriate $C, C'$ and $\gamma$. We omit the details.
\end{proof}

\begin{lemma} \label{lemmaG}
There exist positive constants $C, C'$ and $\gamma$ such that
\[
\pr(\badthree) < Ce^{-C'n^{\gamma}}.
\]
\end{lemma}
Before proving this result, we need to show a few preliminary results.
Furthermore, in this context we work with \textbf{Version} $2$ of the model, introduced in Section \ref{sect:version2}. Lemma \ref{equi12} allows us to switch between events defined in one version to the other.
For $x\in 3\mathbb{Z}$ let $$W^x(\ell)=\frac{\bar{X}^{x}_\ell-\bar{Y}^{x}_\ell}{2}$$ where $\bar{X}^{x}_\ell,\bar{Y}^{x}_\ell$ are defined in Section \ref{sect:version2} \eqref{stack22}. This represents the change in the difference of oil and water particles at $x$ when the $\ell^{th}$ firing takes place from the set $\{x-1,x+1\}.$ Clearly $W^x(\ell)$ has the same distribution as one step of a symmetric lazy random walk. 
Now define
\[
\even(k) =\#\left\{0 \leq j \leq k: \sum_{\ell=1}^{j}W^{x}(\ell)=0\right\}.
\]
Define $\badthree'$ to be the event that there exist three integers $i,j,k$ such that
\begin{itemize}
\item[(i)] $|i|, |j| \leq n^5$,
\item[(ii)] $j-i\geq .1n^{.01}$,
\item[(iii)] $k \in [n^{.5},n^5]$ and
\item[(iv)]
$\displaystyle
\sum_{x \in (i,j), 3|x} \even(.9k)<.01\sqrt{k}(j-i)$. 
\end{itemize}
We now state a standard fact about number of returns to origin for the simple random walk on $\mathbb{Z}$. 

\begin{lemma} \label{srwreturns}
Let $\{S_i\}_{i \geq 0}$ be a lazy simple random walk on $\Z$ started at the origin. Let
\[
\zeros(l)=\#\{i:\ 0 \leq i \leq l \  {\text and } \  S_i=0\}.
\]
Then for all $l$
\[
\pr\bigg(\zeros(.9l)>.1\sqrt{l}\bigg) \geq \frac12.
\]
\end{lemma}
\begin{proof} See Chapter III, Section 5 of \cite{fell}. \end{proof}

\begin{lemma} \label{lemmaK}
There exist positive constants $C, C'$ and $\gamma$ such that
$$\pr(\badthree')
< Ce^{-C'n^{\gamma}}.$$
\end{lemma}

\begin{proof}
Fix $i,j,k$ and $x \in (i,j)$. Then, from Lemma \ref{srwreturns} it follows
\[
\pr\left(\even(.9k)>.1 \sqrt{k}\right)\geq \frac12 .
\]
Therefore we get
\[
\begin{split}
\pr & \left(\sum_{x \in (i,j), 3|x} \even(.9k)<.01\sqrt{k}(j-i)\right) \\
&< \pr \left(\#\left\{x \in (i,j), 3|x \ : \ \even(.9k)>.1 \sqrt{k}\right\} < \frac{1}{10}(j-i)\right)\\
&< ce^{-c' n^{.1}},
\end{split}
\]
where $c,c'$ are positive constants.

As there are at most $4n^{15}$ choices of $i, j$ and $k$
there exists positive constants $c, c'$ and $\gamma'$ such that
$$\pr(\badthree')
<4n^{15} ce^{-c'n^{\gamma'}}$$
so the lemma is true for some choice of $C, C'$ and $\gamma$.
\end{proof}

\begin{pfoflem}{\ref{lemmaG}}
Consider the map $\phi$ defined in Lemma \ref{equi12}. Consider the event $$\phi^{-1}(\badthree).$$ By the measure preserving property of $\phi$
$$\pr(\badthree)=\pr(\phi^{-1}(\badthree))$$ where the two probabilities are in the two different probability spaces mentioned in the statement of Lemma \ref{equi12}.
Note that the event $\phi^{-1}(\badthree)$ implies $\badthree'$. This is because 
(iv) in the definition of $\badthree'$
says $$\displaystyle
\sum_{x \in (i,j), 3|x} \even(.9k)<.01\sqrt{k}(j-i)$$ whereas in the definition of $\badthree$ we have $$\displaystyle
\sum_{x \in (i,j)} \even(.9k)<.01\sqrt{k}(j-i).$$
The proof now follows from the above lemma. 

\end{pfoflem}

\begin{lemma} \label{lemmaI}
There exist positive constants $D, C'$ and $\gamma$ such that
\[
\pr(\badzero),\pr(\badone), \pr( \badtwo), \pr(\badthree), \P(\star)
< De^{-C'n^{\gamma}}.
\]
\end{lemma}
\begin{proof}
$\pr(\badzero)$ and $\pr(\badone)$ are bounded by 
Lemma \ref{aprioridom}. 
$\pr(\badtwo)$ is bounded by Lemma \ref{lemmaF}.
$\pr(\badthree)$ is bounded by Lemma \ref{lemmaG}.
$\P(\star)$ is bounded by Lemma \ref{lemmaA}
\end{proof}

\subsection{Proof of the upper bound}
Thus by \eqref{felix} and lemma \ref{lemmaI}, i.e there exists $\e>0$ such that for large enough $n$ 
\begin{equation}\label{upproof1}\pr(\m)> 1-e^{-n^{\e}}
\end{equation} 
where $G$ was defined in \eqref{mainevent}. 
The proof of the upper bound is hence complete. 
\qed

%\section{Proof of Theorem \ref{lowerbound1}}\label{lowerbound}
\section{Proof of Theorem \ref{lemmaJ} (Lower bound)}\label{lowerbound}
%!TEX root = oil-and-water.tex
The goal of this section is to prove the lower bound in Theorem \ref{lemmaJ}.  Theorem \ref{spread1} will follow from the proof of the lower bound.
Before proceeding to the proofs we state a few standard results about $R_i$, the lazy simple symmetric random walk on $\mathbb{Z}$  whose increments $R_{i+1}-R_i$ are $\pm 1$ with probability $\frac{1}{4}$ each, and $0$ with probability $\frac12$. Let

\begin{equation}\label{supnotation}
M(t)=\sup_{i<t}|R_{i}|.
\end{equation} 

\begin{lemma}\label{mean1}
\begin{itemize}
\item[(i)] Given $\e>0,$ there exists $c=c(\e)>0$ such that
$$\mathbb{P}\left(M(t)>t^{1/2+\e}\right)< e^{-t^{c}}.$$ 
\item[(ii)]Given $\e>0$, there exists $\delta=\delta(\e)>0,$ such that ,  
\begin{equation}\label{tail1}
\P(M(t^{2+\e})<2t)<e^{-t^{\delta}}.
\end{equation}
\item[(iii)]$$ \E(M(t))=\Theta(\sqrt{t}).
$$
\item[(iv)]$$ \E(M(t)^2)=\Theta(t).
$$
\item[(v)]$$\lim_{t\rightarrow \infty}t^{-1/2}\E|R_t|=\sqrt{\frac{1}{\pi}}.$$
\end{itemize}
\end{lemma}
\begin{proof}
For proofs of parts (i)-(iv) we refer the reader to \cite[ Sections 21, 23]{spitzer}. 

By the central limit theorem, $t^{-1/2} R_t \xrightarrow{d} Z$ where $Z \sim N(0,\frac12)$, hence $\E |Z| = \sqrt{1/\pi}$.  Part~(v) follows since the random variables $t^{-1/2} R_t$ are uniformly integrable; see, for instance, \cite[Theorem 3.5]{Bil}.
\end{proof}

\begin{lemma}\label{mean2}
Let 

 $$\bigl\{M^{i}(n^{\frac{4}{3}})\bigr\}_{i=1}^{n^{1/3}}$$
be a sequence of i.i.d.\ random variables with the same law as  $M(n^{\frac{4}{3}})$. Then there exist positive constants $D,\gamma$ such that $$\P\left(\sum_{i=1}^{n^{1/3}}M^{i}\left(n^{\frac{4}{3}}\right)<Dn\right)>1-{e^{-n^{\gamma}}}.$$
\end{lemma}
Proof of Lemma \ref{mean2} is deferred to  Appendix \ref{porse}.
For $x\in \mathbb{Z}$ and a positive integer $i$ we define the variable
\begin{equation}\label{diff2}
D^{x}(i)=\mathbf{1}_{({X}_{i}^x=1)}-\mathbf{1}_{({Y}_{i}^x=1)}
\end{equation}
where the variables $X_{i}^x,Y_{i}^x$ appear in \eqref{set1} in Section \ref{sdes}.
Clearly 
\[
D^{x}(i) = \left\{\begin{array}{cl}
-1 & \mbox{ w.p. } 1/4\\
0 & \mbox{ w.p. } 1/2\\
1 & \mbox{ w.p. } 1/4 .
\end{array}
\right.
\]

Let $C$ be the number appearing in Theorem \ref{lemmaJ}.
For any $x\in \mathbb{Z}$ let 
\begin{align}
\label{supdef11}
S_x & =  \,  \sup_{i\le Cn^{4/3}}\biggl|\sum_{j=0}^i D^{x}(j)\bigg|\\
\label{supdef12}
O_x & =   \sup_{i\le Cn^{4/3}}\biggl|\sum_{j=0}^i \frac{X^{x}_j}{2}\bigg|\\
\label{supdef13}
W_x & =    \sup_{i\le Cn^{4/3}}\biggl|\sum_{j=0}^i \frac{Y^{x}_j}{2}\bigg|
\end{align}
We now discuss briefly how the proof proceeds. We first establish the following gradient bound on the odometer using \eqref{upproof1}: there exists a constant $c$ such that, with high probability, for  $x\in[0,cn^{1/3}]$ we have
$$u(x)-u(x+1)\ge n/2.$$
This in turn implies the theorem since $$0\le u(\lfloor{cn^{1/3}}\rfloor)\le u(0)-\frac{n}{2}\lfloor{cn^{1/3}}\rfloor.$$ 
%that with high probability $$u(0)=\Omega(n^{4/3}),$$
\subsection{Lower bound.}

Recall \eqref{eq:g_t} and that $\tau$ is the stopping time defined in Definition \ref{stoptimedef}.
Thus the total number of particles at a site $x$ at the end of the process is $|g_{\tau}(x)|.$
Also define
\begin{equation}\label{eq:Dx,x+1}
\begin{split}
& D_{x,x+1}=\#\{ \textnormal{ particles moving from }x \textnormal{ to }x+1\},\\
& D_{x+1,x}=\#\{ \textnormal{ particles moving from }x+1 \textnormal{ to }x\}.
\end{split}
\end{equation}
\begin{remark}\label{particleright}
For $x>0$ the difference $D_{x,x+1}-D_{x+1,x} $ is the total number of particles that stop to the right of $x$.
\end{remark}

Since one side of the origin has at least $n/2$ particles at the end of the process  
without loss of generality we assume that 
\begin{equation}\label{wlog1}
\sum_{i=1}^{\infty}|g_{\tau}(i)|\ge n/2.
\end{equation}
Also by definition (see \eqref{diff2}) for any $x\in \mathbb{Z}$, 
\begin{equation}\label{oilwater}
g_{\tau}(x)=\sum_{i=1}^{u(x-1)}D^{x-1}(i)-\sum_{i=1}^{u(x+1)}D^{x+1}(i).
\end{equation}

Recall $\m$  defined in \eqref{mainevent}.
Also given $\e>0$ define the following events 
\begin{eqnarray}
\label{eventdef1}
\n&=&\sum_{x=-\epsilon n^{1/3}}^{\epsilon n^{1/3}}|S_x|\le n/12\\
\label{eventdef2}
\mm&=&2\sum_{x=-\epsilon n^{1/3}}^{\epsilon n^{1/3}}(O_x+W_x)\le n/6.
\end{eqnarray}
where $S_x,O_x,W_x$ are defined in \eqref{supdef11},\eqref{supdef12} and \eqref{supdef13}.
Note that we suppress the dependence on $\e,n$ in the notations for brevity.
\begin{lemma}For small enough $\e$ there exists $c>0$ such that
\begin{equation}\label{lowprob1}
\mathbb{P}(\m\cap \n\cap \mm)\ge 1-e^{-n^{c}}.
\end{equation}

\end{lemma}
\begin{proof}
Follows from  \eqref{upproof1} and Lemma \ref{mean2}. 
\end{proof}
We now state the following lemma establishing a lower bound on the gradient of the odometer function.
\begin{lemma}\label{lowergradient12} Assume \eqref{wlog1}. Then there exists a constant $\e$ such that  $$u(j)-u(k)\ge (k-j)\frac{n}{3}-\frac{n}{6}$$
for all $0\le j\le k\le \e n^{1/3}$ with failure probability at most $e^{-n^{c}}$ for some positive constant $c.$
\end{lemma}
\begin{proof}
Recalling Definition \ref{notationfluc1} $$D_{x,x+1}=u(x)+\Delta^{x}(u(x)).$$
Using Remark \ref{particleright} we have 
\begin{eqnarray*}%\label{eq:diff_Ux}
u(0)+\Delta^{0}(u(0))-u(1)+\Delta^{1}(u(1))& =&\sum_{y=1}^{\infty}|g_{\tau}(y)|.\\
u(1)+\Delta^{1}(u(1))-u(2)+\Delta^{2}(u(2))& =&\sum_{y=2}^{\infty}|g_{\tau}(y)|.\\
\vdots\\
u(m-1)+\Delta^{m-1}(u(m-1))-u(m)+\Delta^{m}(u(m))& =&\sum_{y=m}^{\infty}|g_{\tau}(y)|.\\
\vdots
\end{eqnarray*}
Adding the above from any $0<\ell<m$ to $m$ we get 

\begin{equation}\label{eq:grad}
u(\ell)-u(m)+\Delta^{\ell}(u(\ell))+2\sum_{i=\ell+1}^{m-1}\Delta^{i}(u(i))+\Delta^{m}(u(m))=\sum_{i=\ell+1}^{m}\sum_{y=i}^{\infty}|g_{\tau}(y)|.
\end{equation}
We claim that for any $x\in \mathbb{Z}$ $$|\Delta^{x}(u(x))|\mathbf{1}(\m)\le O_x+W_x$$
where  $O_x,W_x$ are defined in \eqref{supdef12} and \eqref{supdef13}.
This follows from definitions and the observation that 
\begin{eqnarray*}
\Delta^{x}(u(x))&=&\sum_{i=1}^{u(x)}[\mathbf{1}(X^{x}_{i}=1)-1/2]+\sum_{i=1}^{u(x)}[\mathbf{1}(Y^{x}_{i}=1)-1/2]\\
&=& \sum_{i=1}^{u(x)}\frac{X^{x}_{i}}{2}+\sum_{i=1}^{u(x)}\frac{Y^{x}_{i}}{2}.
\end{eqnarray*}
Thus by \eqref{eq:grad}
\begin{equation}\label{represent}
[u(\ell)-u(m)]\mathbf{1}(\m)\ge \sum_{i=\ell+1}^{m}\sum_{y=i}^{\infty}|g_{\tau}(y)|\mathbf{1}(\m)-2\biggl[\sum_{x=\ell}^{m}(O_x+W_x)\biggr]\mathbf{1}(\m).
\end{equation}
Now on the event $\m\cap \n\cap \mm$, $\forall\, j< \e n^{1/3}$
\begin{equation}\label{step1}
\sum_{i=j}^{\infty}|g_{\tau}(i)|\ge n/3.
\end{equation} 
To show this we upper bound $\sum_{i=1}^{j}|g_{\tau}(i)|$.
By \eqref{oilwater} on the event $\m$ we have for all $0<i\le \e n^{1/3},$
\begin{equation}\label{upper12}
|g_{\tau}(i)|\le S_{i-1}+S_{i+1}
\end{equation}
where $S_i$ is defined in \eqref{supdef11}. 
Hence on the event $\m \cap \n \cap \mm$  $$\sum_{i=1}^{j}|g_{\tau}(i)|\le 2\sum_{i=1}^{\e n^{1/3}}S_{i}\le \frac{n}{6}.$$ 

Thus by \eqref{wlog1} we have for all $0<j<\e n^{1/3},$
$$\sum_{i=j}^{\infty}|g_{\tau}(i)|\ge \frac{n}{2}- \sum_{i=1}^{j}|g_{\tau}(i)|\ge \frac{n}{2}-\frac{n}{6}=\frac{n}{3}.$$
 Therefore by \eqref{represent} and \eqref{step1},  on the event $\m \cap \n \cap \mm$ we have
for all $0\le j\le k \le \e n^{1/3}$
\begin{eqnarray}
\label{arg12}
u(j)-u(k)& \ge & \frac{1}{3}(k-j)n -2\biggl[\sum_{x=1}^{k}O_x+W_x\biggr]\\
&\ge & \frac{1}{3}(k-j)n-\frac{n}{6}. 
\end{eqnarray}
Thus we are done.
\end{proof}
The proof of the lower bound is now a corollary.
Since $u(k)\ge 0$, for all large enough $n$ and $j\le \frac{\e n^{1/3}}{2} $ the above implies that 
$$u(j)\ge \frac{1}{8}\e n^{4/3}.$$

Hence by \eqref{lowprob1} it follows 
\begin{equation}\label{finalstep0}
\pr\left(\inf_{0\le j\le \frac{1}{2}\e n^{1/3}}u(j)\le \frac{1}{8}\e n^{4/3}\right)\le e^{-n^{c}}.
\end{equation}
To complete the proof we use the symmetric version of \eqref{eq:grad} to get the following bound: For $j\ge 0$
$$u(0)-u(-j)-\Delta^{0}(u(0))-2\sum_{i=-j+1}^{-1}\Delta^{i}(u(i))-\Delta^{-j}(u(-j))=\sum_{i=-j}^{-1}\sum_{y=-\infty}^{i}|g_{\tau}(y)|$$
The following bound $$\sum_{y=-\infty}^{i}|g_{\tau}(y)|\le 2n$$ is trivial since the total number of particles is $2n.$
Using the above and definition of $\mathcal{G}_2$ we get for all $0<j<\e n^{1/3}$
$$u(0)-u(-j)\le 2jn+\frac{n}{6}.$$
Thus $u(0)\ge \frac{1}{4}\e n^{4/3}$ implies that for all $j\le\frac{1}{16} \e n^{1/3}$ 
\begin{equation}\label{finalstep1}
u(-j)\ge \frac{1}{9} \e n^{4/3}.
\end{equation}
\eqref{finalstep0} and \eqref{finalstep1} now completes the proof.
%The proof now follows Theorem \ref{lemmaJ} and Lemma \ref{lemmaF}. 
\qed

\begin{lemma}\label{ubparticles}There exists a constant $C>0$ such that\\
\begin{equation}\label{part1}
\sup_{x\in \mathbb{Z}}\E(|g_{\tau}(x)|)<Cn^{\frac{2}{3}}.
\end{equation}
Moreover for any $\e>0$ there exists $c>0$
\begin{equation}\label{part2}
\mathbb{P}\left[\sup_{x\in \mathbb{Z}}|g_{\tau}(x)|\ge n^{\frac{2}{3}+\e}\right]\le e^{-n^c}.
\end{equation}
\end{lemma}
\begin{proof}  For any $x\in \mathbb{Z},$ we have 
\begin{equation}\label{step23}
|g_{\tau}(x)|\le [S_{x-1}+S_{x+1}]\mathbf{1}(\m)+2n\mathbf{1}(\m^c)
\end{equation} 
where the event $G$ is defined in \eqref{mainevent}. 
The first term follows from \eqref{upper12} and the second term is obvious since the total number of particles is $2n.$
The proof of \eqref{part1} now follows from $(iii)$ of Lemma \ref{mean1} and Theorem \eqref{upproof1}. Additionally using $(i)$ of Lemma \ref{mean1} we get that for any $x\in \mathbb{Z}$  there exists $c>0$ such that
\begin{equation}\label{pointestimate}
\mathbb{P}(|g_{\tau}(x)|\ge n^{\frac{2}{3}+\e})\le e^{-n^c}.
\end{equation}
Corollary \ref{preliminarysupport} says that with probability at least $1-e^{-n^c}$ for all $|x|\ge n^{5}$
$$|g_{\tau}(x)|=0.$$
\eqref{part2} now follows from \eqref{pointestimate} by union bound over all $x\in [-n^5,n^5].$ 
  
\end{proof}

\begin{remark}\label{lowerboundsupport} For small enough $\e$ and $k=\e n^{1/3},$ by \eqref{lowprob1} and \eqref{step23} $$\mathbb{P}\left[\sum_{i=-k}^{k}|g_{\tau}(i)|\le n\right]\ge 1-e^{-n^c}$$ for some positive constant $c.$ Thus at least $n$ particles are supported outside the interval $[-\e n^{1/3},\e n^{1/3}]$ with failure probability at least $e^{-n^c}.$
\end{remark} 

\subsection{Proof of Theorem \ref{spread1}}

Let $x=n^{1/3+\e}$. Under the assumption that there are at least $n^{1-\epsilon/2}$ many particles to the right of $x$, for all $\ell \le x,$ $$\sum_{i=\ell}^{\infty}|g_{\tau}(i)|\ge n^{1-\e/2}.$$
Recalling \eqref{represent} we have
$$(u(0)-u(x))\mathbf{1}(\m)\ge \sum_{i=1}^{x}\sum_{y=i}^{\infty}|g_{\tau}(y)|\mathbf{1}(\m)-2\biggl[\sum_{i=1}^{x}(O_i+W_i)\biggr]\mathbf{1}(\m).$$
Now by part $(i)$ of Lemma \ref{mean1} and union bound over $1\le \ell \le x$ there exists a $c>0$ such that with probability at least $1-e^{-n^{c}},$  $$\sum_{\ell =1}^{x}[O_\ell+W_\ell]=O(n^{1+2\e}).$$  
Thus on the event that there are at least $n^{1-\epsilon/2}$ many particles to the right of $x$ we have $$u(0)-u(x)\ge xn^{1-\e/2} -O(n^{1+\e}),$$ except on a set of measure at most $e^{-n^{c}}$.  However this implies that $$u(0)\ge n^{4/3+\e/2}-O(n^{1+\e}).$$  Hence  by the upper bound in Theorem \ref{lemmaJ} we conclude that the probability of the event that there are at least $n^{1-\epsilon/2}$ many particles to the right of $x=n^{1/3+\e}$ is less than $e^{-n^{c}}$ for some positive $c>0$. The argument for $x=-n^{1/3+\e}$ is symmetric and we omit the details. Thus we are done.
\qed

\section{Scaling limit for the odometer}\label{prop}
%!TEX root = oil-and-water.tex

The goal of this section is to prove Theorem \ref{conditionalscalingthm}.   The first step will be to show that Conjecture \ref{ass1} implies some regularity of the limiting function: we will argue that $w(\cdot)$ is decreasing and three times differentiable on the positive real axis. Moreover it is the solution of the boundary value problem
\begin{eqnarray*} 
w''=&\sqrt{\frac{2}{\pi}w}\\
\lim_{h\rightarrow 0^+} \frac{w(h)-w(0)}{h}=&-1\\
\lim_{h\rightarrow \infty} w(h)=& 0.
\end{eqnarray*} 
At this point, Theorem \ref{conditionalscalingthm} follows by identifying an explicit solution to the above problem and arguing that it is the unique solution.

\subsection{Properties of the expected odometer}
We first make some easy observations about the expected odometer function, denoted  by
\[
\tilde u(x):=\E(u(x)).
\]
Existence of $\tilde u(x)$ follows from $(ii)$ of Lemma \ref{aprioridom} which says that the stopping time of the process $\tau$ has finite expectation and clearly for all $x\in \mathbb{Z}$ $$u(x)\le \tau.$$

Recall the notation $D_{x,x+1}$ from \eqref{eq:Dx,x+1}.
\begin{lemma}\label{lemma:expected_diff_odom}For $x\in \mathbb{Z}$ $$\tilde u(x) = \E (D_{x,x+1}) = \E( D_{x,x-1}).$$ Moreover we have,
\begin{eqnarray}
\tilde u(x)-\tilde u(x+1) \le n \,\,\text{ for }x\ge 0\\
\tilde u(x)-\tilde u(x-1)   \le n \,\,\text{ for }x\le 0.
\end{eqnarray}
\end{lemma}
Clearly every time there is an emission at a site $x\in \mathbb{Z}$, on average one particle moves to $x-1$ and another to $x+1$. Hence, informally for every $x\in\mathbb{Z}$, $\tilde u(x)$ is the expected number of particles emitted from $x$ that go to $x-1$ or $x+1$. 
\begin{proof} 
Wlog assume $x\ge 0$.
By definition we have
$$D_{x,x+1}=\sum_{k=1}^{\infty}( \mathbf{1}({X^{x}_k=1})+ \mathbf{1}({Y^{x}_k=1}))\mathbf{1}(u(x)>k-1).$$
Let $\tau_k^-=\tau_k^-(x)$ be the stopping time such that at time $\tau_k^-+1$ the $k^{th}$ pair is emitted from $x.$
Now clearly $ \mathbf{1}({X^{x}_k=1})+ \mathbf{1}({Y^{x}_k=1})$ is independent of the filtration $\mathcal{F}_{\tau_k^-\wedge \tau}$. Also $\mathbf{1}(u(x)>k-1)$ is measurable with respect to $\mathcal{F}_{\tau_k^-\wedge \tau}$. Hence $\E(  \mathbf{1}({X^{x}_k=1})+ \mathbf{1}({Y^{x}_k=1})|\mathcal{F}_{\tau_k^-\wedge \tau})=1.$
Thus, 
\begin{align*}\E(D_{x,x+1})&=\sum_{k=1}^{\infty}\E\left[\E\left( \mathbf{1}({X^{x}_k=1})+ \mathbf{1}({Y^{x}_k=1})\mathbf{1}(u(x)>k-1)|\mathcal{F}_{\tau_k^-\wedge \tau}\right)\right]\\
&=\sum_{k=1}^{\infty}\E(\mathbf{1}(u(x)>k-1))\\
&=\E(u(x)).
\end{align*}

Similarly $$\E(D_{x,x-1})=\E(u(x)).$$

Since the total number of particles is $2n$, using the symmetry of the process with respect to the origin one can conclude that the expected number of particles at the end of the process  to the right of $x$ is at most $n$. Hence using  Remark \ref{particleright} we have
\[
\E (D_{x,x+1}-D_{x+1,x} )\le n.
\]
Thus we are done.
\end{proof}

\begin{lemma}\label{expected_odometer} $\tilde u(x)$ satisfies the following properties:
\begin{itemize}
\item[(i)] $\tilde u(x)$ is an even function;
\item[(ii)] restricted to $\mathbb{Z}_{+}$, $\tilde u(x)$ is strictly decreasing;
\item[(iii)] for every $x\neq 0$ 
$$\Delta \tilde u(x)>0$$ where $\Delta$ is the discrete Laplacian i.e.
\[
\Delta \tilde u(x)=\tilde u(x-1)+\tilde u(x+1)-2\tilde u(x).
\]
\end{itemize}
\end{lemma}

\begin{proof}
The proof of $(i)$ follows from the symmetry of the process with respect to the origin.
\smallbreak

To prove $(ii)$ we first recall the definition of the function $g_t$ from (\ref{eq:g_t}). For any $x\ge0$ by Lemma~\ref{lemma:expected_diff_odom}, the difference $\tilde u(x)- \tilde u(x+1)$ is the expected number of particles that stop to the right of $x$:
\begin{equation}\label{eq:g_tau_x1}
\tilde u(x)-\tilde u(x+1)= \E\left(\sum_{y=x+1}^{\infty}|g_{\tau}(y)|\right).
\end{equation}
Since the quantity on the right hand side is nonnegative we have that $\tilde u(x)$ is non increasing. To see that it is strictly decreasing, we make the simple observation that given any $x>0$, with positive probability all the $2n$ particles stop somewhere to the right of $x$. In other words,
\[
\mathbb{P}\left(\sum_{y=x+1}^{\infty}|g_{\tau}(y)|=2n\right)>0.
\]
Hence $\E\left(\sum_{y=x+1}^{\infty}|g_{\tau}(y)|\right)>0$, implying the statement.
\smallbreak
To prove $(iii)$ we see that by \eqref{eq:g_tau_x1} 
\begin{equation}\label{laplaceg}
%\begin{split}
\Delta \tilde u(x) = \tilde u(x-1)-2\tilde u(x)+\tilde u(x+1)
= \E \left (\sum_{y=x}^{\infty}|g_{\tau}(y)|-\sum_{y=x+1}^{\infty}|g_{\tau}(y)| \right ) = \E(|g_{\tau}(x)|).
\end{equation}
By using similar reasoning as in the proof of $(ii)$, we see that there is positive chance that $|g_{\tau}(x)|>0$. Hence $\Delta \tilde u(x)> 0$.
\end{proof}

\subsection{The differential equation $w'' = \sqrt{2w/\pi}$}\label{sect:final_step}
In this section we work toward the proof of Theorem \ref{conditionalscalingthm}.
We recall Conjecture \ref{ass1}  stated in the introduction.
Note that we have not assumed a priori that $w$ is continuous. Proving this is our first order of business.

\begin{lemma}\label{someprop}
$w$ is continuous, (in fact, $1$-Lipschitz) on $\R$. Moreover it is non-increasing on the positive real axis.
\end{lemma}

\begin{proof}
 $(i)$ of Lemma~\ref{expected_odometer} implies that $w$ is an even function. Hence it suffices to prove that $w$ is $1$-Lipschitz on $[0,\infty)$.
By Lemma \ref{lemma:expected_diff_odom}, for any $x, k \in \Z_{\geq 0}$ we have
	\[  0 \leq \tilde u(x) - \tilde u(x + k) = \sum_{j=0}^{k-1} [\tilde u(x+j) - \tilde u(x+j+1)] \leq kn. \]
Now let $x = \lfloor{ n^{1/3}\xi} \rfloor$ and $k=\lfloor{n^{1/3}(\xi+h)}\rfloor-\lfloor{n^{1/3}\xi}\rfloor$. Dividing by $n^{4/3}$ and taking $n \to \infty$, we obtain
	\[ 0 \leq w(\xi) - w(\xi+h) \leq h \]
thus completing the proof of the lemma.
\end{proof}
Recall the set of random variables $$X^{x}_i,Y^{x}_i$$ defined in \eqref{set1}. 
To go further, we define the following quantities:
 For $y=x\pm1$
\[
\begin{split}
& O_{x,y}^{k}:=\sum_{i=0}^k \mathbf{1}(X^{x}_i=\pm 1),\\
& W_{x,y}^{k}:=\sum_{i=0}^k \mathbf{1}(Y^{x}_i=\pm 1).
\end{split}
\]
In other words,
\[
\begin{split}
& O_{x,y}^{k}:=\#\{\textnormal{oil particles sent from }x\textnormal{ to }y\textnormal{ within the first }k\textnormal{ firings at }x\},\\
& W_{x,y}^{k}:=\#\{\textnormal{water particles sent from }x\textnormal{ to }y\textnormal{ within the first }k\textnormal{ firings at }x\}.
\end{split}
\]

Thus

\begin{equation}\label{version1}	
		g_\tau(x) = O_{x-1,x}^{u(x-1)}+O_{x+1,x}^{u(x+1)}-W_{x-1,x}^{u(x-1)}-W_{x+1,x}^{u(x+1)}
		\end{equation} 
	where $g_t(\cdot)$ was defined in \eqref{eq:g_t}.
Consider now the analogous expression using a deterministic portion of each stack (recall $\tilde u(x) = \E u(x)$):
\begin{equation}\label{expversion}
		\tilde g_\tau(x) := O_{x-1,x}^{\lfloor\tilde u(x-1)\rfloor}+O_{x+1,x}^{\lfloor\tilde u(x+1)\rfloor}-W_{x-1,x}^{\lfloor\tilde u(x-1)\rfloor}-W_{x+1,x}^{\lfloor\tilde u(x+1)\rfloor}. 
		\end{equation}
Because $\tilde u$ is deterministic the four terms on the right side are independent. Moreover each term $O_{x,y}^k$ and $W_{x,y}^k$ for $|x-y|=1$ is a sum of $k$ independent Bernoulli$(1/2)$ random variables. So the right side is a sum of $\lfloor \tilde u(x-1)\rfloor+ \lfloor\tilde u(x+1)\rfloor$ independent random variables with the same law as a single step of a lazy symmetric random walk as defined in \eqref{diff2}. Setting $x = \lfloor{n^{1/3} \xi\rfloor}$ for a real number $\xi >0$, by $(ii)$ in Conjecture \ref{ass1} we have 
	\begin{eqnarray*}
\lim_{n\rightarrow \infty} \frac{\lfloor\tilde u(x-1)\rfloor+ \lfloor\tilde u(x+1)\rfloor}{n^{4/3}}
= \lim_{n\rightarrow \infty} \frac{2 \tilde u(x)}{n^{4/3}}
& = & 2 w(\xi). 
\end{eqnarray*}
This is because by Lemma \ref{lemma:expected_diff_odom}, $|\tilde u(x)-\tilde u(x+1)|$ and $|\tilde u(x-1)-\tilde u(x)|$ are both less than $n$.
As $n\rightarrow \infty $ by the central limit theorem, since each variable in \eqref{diff2} has variance $1/2$, we obtain 
	\begin{equation}\label{expclt}
	 n^{-2/3} \tilde g_\tau(x) \stackrel{d}{\longrightarrow} N(0,w(\xi)). 
	\end{equation}
	By Lemma \ref{mean1} $(v)$ we also have
    \begin{equation}\label{exp1moment}
	 n^{-2/3} \E |\tilde g_\tau(x)|\, {\longrightarrow}\, \sqrt {\frac{2}{\pi} w(\xi)}. 
	\end{equation}
Next we observe that under $(i)$ in Conjecture \ref{ass1} , the same kind of central limit theorem holds for $g_\tau$ itself.

\begin{lemma} 
\label{l.clt} Let $\xi\ge 0$. For $x = \lfloor{n^{1/3} \xi}\rfloor$, we have as $n \to \infty$
\begin{itemize}
\item[(i)]
	\[ n^{-2/3} g_\tau(x) \stackrel{d}{\longrightarrow} N(0,w(\xi))\]
\item[(ii)]\[n^{-2/3}\E |g_\tau(x)|\longrightarrow\sqrt {\frac{2}{\pi} w(\xi)}.\] 	
\end{itemize}
\end{lemma}
\begin{remark}\label{ptwise12}
 $(ii)$ along with \eqref{laplaceg} implies 
\begin{equation}\label{l.discrete2ndderiv}
\lim_{n\rightarrow \infty}\frac{\Delta \tilde u(x)}{n^{2/3}} \to \sqrt {\frac{2}{\pi} w(\xi)}.
\end{equation}
\end{remark}
To prove Lemma \ref{l.clt} we need the next two results.
\begin{lemma}\label{momentbound} 
 $\E\left[\tau\mathbf{1}(\tau>n^5)\right]=O(1)$ where $\tau $ is defined in \eqref{stoptimedef}. 
\end{lemma}

\begin{corollary}\label{boundexp}There exists a constant $C_1>0$ such that $$\sup_{x\in \mathbb{Z}}\tilde u(x)\le C_1n^{4/3}$$ where $\tilde{u}=\E(u).$
\end{corollary}
The proofs of the above two results are deferred to Appendix \ref{porse}.

\subsubsection{Proof of Lemma \ref{l.clt}.}
By \eqref{expclt} and \eqref{exp1moment} it suffices to show 
\begin{equation}\label{argument}
\lim _{n\rightarrow 0}n^{-2/3} \E|g_\tau(x) - \tilde g_\tau(x)|= 0.
\end{equation}  
Referring to the definitions of $g$ and $\tilde g$ in \eqref{version1} and \eqref{expversion} respectively, this will be accomplished if we show that as $n\rightarrow \infty$ for $y=x\pm 1$ 
	\begin{align*}  n^{-2/3} \E\left|\left(O_{y,x}^{\lfloor\tilde u(y)\rfloor}-W_{y,x}^{\lfloor\tilde u(y)\rfloor}\right)- \left( O_{y,x}^{u(y)}-W_{y,x}^{u(y)}\right)\right| 
			       \end{align*}
tend to $0$. For $y=x\pm1$ the above differences have identical distributions. Hence it suffices to look at any one. The quantity $\left(O_{y,x}^{\tilde u(y)}-W_{y,x}^{\tilde u(y)}\right)- \left( O_{y,x}^{u(y)}-W_{y,x}^{u(y)}\right)$ is a sum of $$N = |\lfloor {\tilde u(y)}\rfloor - u(y)|$$ independent random variables $X_1,\ldots,X_N$ with the same law as in \eqref{diff2}.
By Conjecture \ref{ass1} (i), $N/n^{4/3} \to 0$ in distribution. Fix $\e\ge 0$.
Let 
\begin{itemize}
\item $\displaystyle{A_1=\mathbf{1}(N \le \e n^{4/3})\sup_{1\le m\le \e n^{4/3} } |\sum_{i=1}^{m}X_i|}$\\
\item $\displaystyle{A_2=\mathbf{1}(N \ge \e n^{4/3})\mathbf{1}(\m)\sup_{1 \le m\le (C+C_1)n^{4/3} } |\sum_{i=1}^{m}X_i|}$\\
\item $\displaystyle{A_3=2n^5\mathbf{1}(\m^c)\mathbf{1}(\tau \le n^5)}$\\ 
\item $A_4=2\tau \mathbf{1}(\m^c) \mathbf{1}(\tau>n^5)$
\end{itemize}
where  $C$ and $C_1$ are the constants appearing in the statement of Theorem \ref{lemmaJ} and Corollary \ref{boundexp} respectively, $\tau$ is defined in \eqref{stoptimedef} and $\m$ is defined in \eqref{mainevent}.
We now claim that
\begin{equation}\label{fourpart}
\left|\left(O_{y,x}^{\lfloor\tilde u(y)\rfloor}-W_{y,x}^{\lfloor\tilde u(y)\rfloor}\right)- \left( O_{y,x}^{u(y)}-W_{y,x}^{u(y)}\right)\right| \le A_1+A_2+A_3+A_4.
\end{equation}
The first two terms correspond to the cases
\begin{itemize}
\item $N \le \e n^{4/3}$\\
\item $\{N \ge \e n^{4/3}\}\cap \m$.
\end{itemize}
For the last two terms we use the naive bound that 
\begin{equation}\label{naivebound}
|\sum_{i=1}^{N}X_i|\le N \le \sup_{x\in \mathbb{Z}}\tilde u(x)+\sup_{x\in \mathbb{Z}}u(x) \le C_1n^{4/3}+\tau 
\end{equation}
where the last inequality uses Corollary \ref{boundexp}. Using the above bound and looking at the events  $\{\tau \le n^5\}\cap \m^c$ and $\{\tau > n^5\}\cap \m^c$ gives us \eqref{fourpart}.

\begin{itemize}
\item By Lemma \ref{mean1} $(iii),$ $\E(A_1)=O(\sqrt{\e}n^{2/3})$\\
\item By Cauchy-Schwarz inequality and Lemma \ref{mean1} $(iv)$ $$\E(A_2)=O(n^{2/3})\sqrt{P(N \ge \e n^{4/3})}$$
\item $\E(A_3+A_4)=O(1)$ by Theorem \ref{lemmaJ} and Lemma \ref{momentbound} respectively.
\end{itemize}
Thus for any $\e>0$ $$\E|g_\tau(x) - \tilde g_\tau(x)|\le \sum_{i=1}^{4}\E(A_{i})=n^{2/3}\Bigl(O\bigl(\sqrt{\e}+\sqrt{P(N \ge \e n^{4/3})}\bigr)\Bigr).$$
Hence \eqref{argument} follows using the above and Conjecture \ref{ass1} $(i)$  ($N/n^{4/3}$ goes to $0$ in distribution) and we are done.
\qed \\
\begin{remark}\label{unif14}
Note that we actually prove \eqref{argument} uniformly over $x$ i.e.
\begin{equation*}
\lim _{n\rightarrow 0}\sup_{x\in\mathbb{Z}}n^{-2/3} \E|g_\tau(x) - \tilde g_\tau(x)|= 0.
\end{equation*}  

\end{remark}

We now prove an uniform version of \eqref{l.discrete2ndderiv}.
\begin{lemma}\label{uniform} Given $\e>0$ and $x<y$ such that $w(x),w(y)>0$, for large enough $n$, 
$$\sup_{\lfloor{x}n^{1/3}\rfloor \le j \le  \lfloor{y}n^{1/3}\rfloor} \left|\frac{\Delta{\tilde u(j)}}{n^{2/3}}-\sqrt {\frac{2}{\pi} w(\frac{j}{n^{1/3}})}\right|\le \e$$ 
\end{lemma}
\begin{proof}Since $w$ is continuous by Lemma \ref{someprop} and hence uniformly continuous on $[x,y]$, for $\e>0$ there exists real numbers  
$$x=x_0<x_1<\ldots <x_k=y$$ such that 
\begin{eqnarray*}
\sup_{1\le i \le k}(x_{i}-x_{i-1})& \le & \e \\
\sup_{1\le i \le k}(w(x_{i-1})-w(x_{i}))& \le &\e.
\end{eqnarray*}
 By Conjecture \ref{ass1} $(ii)$ and \eqref{l.discrete2ndderiv} we have for large enough $n$
\begin{eqnarray}
\label{pointodom}
\sup_{1\le i \le k} |\frac{{\tilde u(\lfloor n^{1/3}x_i\rfloor)}}{n^{4/3}}- w(x_i)|& \le & \e \\
\label{pointlapl} 
\sup_{1\le i \le k} |\frac{\Delta{\tilde u(\lfloor n^{1/3}x_i\rfloor)}}{n^{2/3}}-\sqrt {\frac{2}{\pi} w(x_i)}| & \le & \e.
\end{eqnarray}
Now for any $\lfloor{x}n^{1/3}\rfloor \le j \le  \lfloor{y}n^{1/3}\rfloor$ find $0\le i < k$ such that $$\lfloor{x_{i}}n^{1/3}\rfloor\le j \le \lfloor{x_{i+1}}n^{1/3}\rfloor .$$ 
Clearly it suffices to show,
 $$ \left|\Delta \tilde u(j)-\Delta \tilde u(\lfloor{x_{i}}n^{1/3}\rfloor)\right| =O(\sqrt{\e} n^{2/3}).$$ or by \eqref{laplaceg} 
 $$\left|\E|g_{\tau}(j)|-\E|g_{\tau}(\lfloor{x_{i}}n^{1/3})|\right| =O(\sqrt{\e} n^{2/3})$$
Notice that by \eqref{argument}  and Remark \ref{unif14} we have 
\begin{eqnarray}
\label{unif1}
\left|\E|g_{\tau}(j)|-\E|\tilde{g}_{\tau}(j)|\right|&=&o(n^{2/3})\\
\label{unif11}
\left|\E|g_{\tau}(\lfloor{n^{1/3}x_i}\rfloor)|-\E|\tilde{g}_{\tau}(\lfloor{n^{1/3}x_i}\rfloor)\right|&=&o(n^{2/3})
\end{eqnarray}
Hence it suffices to show 
\begin{eqnarray}
\label{suff1}
\bigl|\E|\tilde{g}_{\tau}(j)|- \E|\tilde{g}_{\tau}(\lfloor{n^{1/3}x_i}\rfloor)|\bigr|&=& O(\sqrt{\e} n^{2/3})\\
\label{suff2}
\bigl|\E|\tilde{g}_{\tau}(\lfloor{n^{1/3}x_{i+1}}\rfloor)|-\E|\tilde{g}_{\tau}(j)|\bigr|&=& O(\sqrt{\e} n^{2/3}).
\end{eqnarray}

By \eqref{pointodom}
\begin{eqnarray*}
\tilde u(j)- \tilde u(\lfloor n^{1/3}x_i\rfloor) &\le &\e n^{4/3}\\
\tilde u(\lfloor n^{1/3}x_{i+1}\rfloor )-\tilde u(j) &\le &\e n^{4/3}.
\end{eqnarray*}
Now  by \eqref{expversion} the quantities on the left hand side of \eqref{suff1} and \eqref{suff2} are absolute values of a lazy symmetric  random walk run for time $\tilde u(j)- \tilde u(\lfloor n^{1/3}x_i\rfloor)$ and $\tilde u(\lfloor n^{1/3}x_{i+1}\rfloor )-\tilde u(j)$  respectively.
 The result now follows  from Lemma \ref{mean1} $(v)$.
 \end{proof}

\begin{corollary}\label{integrability}
$$\int_{0}^{\infty}\sqrt{\frac{2}{\pi}w(\zeta)}d\zeta \le 1$$
which in particular implies
$$\lim_{x\rightarrow \infty }w(x)=0.$$
\end{corollary}
\begin{proof} 

By Lemma \ref{lemma:expected_diff_odom} 
\begin{equation} \label{contrarg}
\E \left (\sum_{y=1}^{\infty}|g_{\tau}(y)|\right)= \sum_{y=1}^{\infty}\Delta \tilde u(y)=\tilde{u}(0)-\tilde{u}(1)\le n. 
\end{equation}
Thus for any positive number $A$
$$\sum_{y=1}^{\lfloor An^{1/3}\rfloor}\Delta \tilde u(y)\le n.$$
By Lemma \ref{uniform} and aproximation of integral by Riemann sum  we have
$$1\ge \lim_{n\rightarrow \infty}\sum_{y=1}^{An^{1/3}}\frac{1}{n^{1/3}}\frac{\Delta \tilde u(y)}{n^{2/3}}=\int_{0}^{A}\sqrt{\frac{2}{\pi}w(\zeta)}d\zeta .$$
Since $w$ is non negative it follows that $$\int_{0}^{\infty}\sqrt{\frac{2}{\pi}w(\zeta)}d\zeta\le 1.$$
By Lemma \ref{someprop} $w$ is non increasing, hence this implies that $$\lim_{x\rightarrow \infty} w(x)=0. \qedhere $$ 
\end{proof}
\begin{remark}
Note that we assumed only convergence of $\tilde u$ in Conjecture \ref{ass1} $(ii)$ but were able to use a special feature of the oil and water model (namely, the identity $\Delta \tilde u(x) = \E |g_\tau(x)|$) to obtain something stronger, convergence of the discrete Laplacian $\Delta \tilde u$. 
\end{remark}
Next we use this to argue that the scaling limit $w(\xi)$ is actually a twice differentiable function of $\xi>0$.

For any $\e>0$ by Conjecture \ref{ass1} $(i)$ and the above corollary we can choose $L$ large enough so that for large enough $n$, $$\sup_{|\xi|\ge L} \tilde u(\lfloor n^{1/3}\xi\rfloor) < \eps^2 n^{4/3}$$ for all $|\xi| > L.$
\begin{lemma} \label{l.pioneers}
Given $\e>0$ let $L$ be as chosen above. Then
$$\displaystyle{\sum_{|x| > n^{1/3}L} \Delta \tilde u(x) \leq \eps^{1/2}} n.$$
\end{lemma}

\begin{proof}
Since the $n^{1/3} \eps$ differences $\tilde u(x)- \tilde u(x+1)$ for $x=\lfloor n^{1/3}L\rfloor, \ldots, \lfloor n^{1/3}(L+\eps)\rfloor-1$ are nonnegative and sum to at most $\tilde u(\lfloor n^{1/3}L\rfloor)\le\eps^2 n^{4/3}$, the smallest of them (which is the last one) must be at most $\eps n$.  Therefore
    \[
	 \sum_{|x| > n^{1/3}(L+\eps)} \Delta \tilde u(x) 
	%=\tilde u(n^{1/3} (L+h)) - \tilde u(n^{1/3}(L+h)+1) 
	\leq \eps n. \]
	Now the fact that 
	\[ \sum_{x =\lfloor n^{1/3}L\rfloor}^{\lfloor n^{1/3}(L+\eps)\rfloor} \Delta \tilde u(x) 
	%=\tilde u(n^{1/3} (L+h)) - \tilde u(n^{1/3}(L+h)+1) 
	\leq O(\eps) n. \] follows from Lemma \ref{ubparticles} and the fact $\Delta \tilde u(x)=\E|g_{\tau}(x)|.$ 
Combining the above two results the proof follows.
\end{proof}

\begin{lemma}\label{diffposre}
$w$ is differentiable on the positive real line, and for any $\xi>0$
	\begin{equation} \label{e.w'} w'(\xi) = -\int_{\xi}^\infty \sqrt {\frac{2}{\pi} w(\zeta)} \,d\zeta. \end{equation}
\end{lemma}

\begin{proof} 
By summation by parts, for positive integers $x,k$
	\begin{equation}\label{byparts}\frac{1}{kn}[\tilde u(x) - \tilde u(x+k)] = \frac{1}{kn}\sum_{j=1}^\infty \min(j,k) \Delta \tilde u (x+j). 
	\end{equation}
For positive real numbers $\xi,L,h$   
    let $x=\lfloor \xi n^{1/3}\rfloor,Z = \lfloor Ln^{1/3}\rfloor,k=\lfloor hn^{1/3}\rfloor$ and consider the first part of the sum in (\ref{byparts}) 
    \[ \frac{1}{h} \sum_{j=1}^Z\frac{1}{n^{\frac{1}{3}}}\frac{\min(j,k)}{n^{\frac{1}{3}}} \frac{\Delta \tilde u (x+j)}{n^{\frac{2}{3}}}.
\]
Now given $\delta>0,$ by Lemma \ref{uniform} for large enough $n$ 
        \[\left| \frac{1}{h} \sum_{j=1}^Z\frac{1}{n^{\frac{1}{3}}}\frac{\min(j,k)}{n^{\frac{1}{3}}} \frac{\Delta \tilde u (x+j)}{n^{\frac{2}{3}}}
- \frac{1}{h} \sum_{j=1}^Z\frac{1}{n^{\frac{1}{3}}}\min(\frac{j}{n^{\frac{1}{3}}},h) \sqrt{\frac{2}{\pi} w(\xi+\frac{j}{n^{\frac{1}{3}}})}
\right| \le \delta L.\]
Notice that $$\frac{1}{h} \sum_{j=1}^Z\frac{1}{n^{\frac{1}{3}}}\min(\frac{j}{n^{\frac{1}{3}}},h) \sqrt{\frac{2}{\pi} w(\xi+\frac{j}{n^{\frac{1}{3}}})}$$ is a Riemann sum approximation of the integral $$\frac{1}{h}\int_{\xi}^{L+\xi} \min(\zeta-\xi, h) \sqrt {\frac{2}{\pi} w(\zeta)} \,d\zeta.$$ Thus as $n$ goes to  $\infty$ we see that   
\[\frac{1}{h} \sum_{j=1}^Z\frac{1}{n^{\frac{1}{3}}}\frac{\min(j,k)}{n^{\frac{1}{3}}} \frac{\Delta \tilde u (x+j)}{n^{\frac{2}{3}}} \to \frac{1}{h}\int_{\xi}^{L+\xi} \min(\zeta-\xi, h) \sqrt {\frac{2}{\pi} w(\zeta)} \,d\zeta. \]
Fixing $\e>0$ and choosing the same $L$ as in the statement of Lemma \ref{l.pioneers} we get that the sum of the remaining terms in \eqref{byparts} $\displaystyle{\frac{1}{kn}\sum_{j=Z+1}^\infty}$ is at most $ \eps^{1/2} $ by Lemma~\ref{l.pioneers}. Hence as $n \rightarrow \infty$ we get from \eqref{byparts} 
\[	
	\frac{w(\xi)-w(\xi+h)}{h}= \frac{1}{h} \int_{\xi}^{L} \min(\zeta-\xi, h) \sqrt {\frac{2}{\pi} w(\zeta)} \,d\zeta + O({\e}^{1/2}).
\]
Sending $h$ to $0$ followed by $\e$ to $0$ ($L$ to $\infty$) we are done.
\end{proof}
%Taking $h\to 0$, the right side tends to $\int_{\xi}^{L} \sqrt {\frac{2}{\pi} w(\zeta)} \,d\zeta$.

The right side of \eqref{e.w'} is manifestly a differentiable function of $\xi$, so we obtain the following.

\begin{corollary}
Under Conjecture \ref{ass1}, the function $w$ restricted to the positive real axis is twice continuously differentiable and obeys the differential equation \begin{equation} \label{e.w''} w'' = \sqrt{\frac{2}{\pi} w}. \end{equation}
\end{corollary}

\begin{lemma}\label{valuea}$w$ is compactly supported. Moreover on the positive region of support  $$w(x)=\left(-\frac{1}{4}\left(\frac{32}{9\pi}\right)^{1/4}x +b\right)^{4}$$ for some $b>0$. 
\end{lemma}
Before proving the above we quote the well known Picard existence and uniqueness result for ODE's.
\begin{theorem}\cite[Theorem 8.13]{kel}\label{picard} Consider an initial value problem (IVP)
\begin{eqnarray}
y'(x)&=&f(y(x),x)\\
y(x_0)&= & y_0
\end{eqnarray}
with the point $(x_0,y_0)$ belonging to some rectangle $(a,b)\times (A,B)$ i.e
$a < a_0 <b$ and $A< y_0 < B.$
Also assume that $f$ is $M-$ Lipchitz for some $M\ge 0$ i.e.
$$|f(z,x)-f(w,x)|\le M|z-w|$$ for all $x\in (a,b)$ $z,w \in (A,B).$
Then there exists a $h=h(x_0,y_0,M)>0$ such that 
\begin{itemize}
\item \textbf{Existence:} There exists a solution to the IVP on the interval $(x_0-h,x_0+h)$.
\item \textbf{Uniqueness:} Any two solutions of the IVP agree on the interval $(x_0-h,x_0+h)$. 
\end{itemize}
\end{theorem}
 
\textit{Proof of Lemma \ref{valuea}}.  
Multiplying \eqref{e.w''} by $w'$ on both sides we get $$w'w'' = \sqrt{\frac{2}{\pi} w}w'.$$
Integrating both sides from $\xi$ to $\infty$ and using the fact that $\displaystyle{\lim_{x\to \infty}w(x)=\lim_{x\to \infty}w'(x)=0}$ (from Corollary \ref{integrability} and Lemma \ref{diffposre}) and that $w'$ is non positive
we see that $w(x)$ satisfies the first order ODE \begin{equation}\label{ode}f' = - \left(\frac{32}{9\pi}\right)^{1/4}f^{3/4}.
\end{equation}
Now suppose $w$ is positive on the entire real axis. 
Given any $z\in \mathbb{R}_{+}$ then $w(z)$ and $w'(z)$ are both non zero. Thus we can find $a,b$ such that 
\begin{eqnarray*}
(az+b)^{4}=w(z) \\
4a(az+b)^{3}=w'(z)  
\end{eqnarray*}
By \eqref{ode} $$4a=- \left(\frac{32}{9\pi}\right)^{1/4}.$$ 
Because of the particular choice of $a$ and $b$ the function $(ax+b)^{4}$ also satisfies (\ref{ode}).
Now since $w(z)$ and $w'(z)$ are both non zero  the function $w^{3/4}(z)$ is Lipchitz in a neighborhood of $z$. Hence by Theorem \ref{picard}   ODE (\ref{ode}) has an unique solution in some neighborhood of $z$. Thus the functions $w(x)$ and $(ax+b)^4$ are equal in a neighborhood of $z$.
Now looking at the biggest interval $I$ containing $z$  such that  $w(x)=(ax+b)^{4}$
on $I$ we conclude that $w(x)=(ax+b)^{4}$ on $\mathbb{R}_{+}\cap supp(w).$ In particular since $(ax+b)^{4}$ is positive only on a compact set this implies that $w(x)$ has compact support.
\qed

Now we find the value for $b$ which completely determines $w.$

\begin{lemma}\label{valueb}$$-4ab^3=\lim_{h\to 0^{+}}\frac{w(0)-w(h)}{h}=1.$$
In particular $$b=\left(\frac{9\pi}{32}\right)^{1/12}.$$
\end{lemma}
\begin{proof}
That $-4ab^3=\lim_{h\to 0^{+}}\frac{w(0)-w(h)}{h}$ follows from Lemma \ref{valuea}. 
To see that the quantity equals $1$ fix $h>0.$
Consider the telescopic sum
$$
 \tilde u(0)-\tilde u(hn^{1/3})= \sum_{i=0}^{hn^{1/3}}\tilde{u}(i)-\tilde{u}(i+1).
 $$

Now $\tilde u(i)-\tilde u(i+1)$ is the expected number of particles on the right of $i$ by Lemma \ref{lemma:expected_diff_odom}.
By  symmetry of the process about the origin and Lemma \ref{ubparticles}
$$\sum_{x>0}|g_{\tau}(x)|=n-O(n^{2/3}).$$    
Moreover for any $i>0$
$$\tilde u(i)-\tilde u(i+1)=n-O(n^{2/3})i.$$
Summing over $i$ we get  $$ \tilde u(0)-\tilde u(\lfloor{hn^{1/3}}\rfloor)=hn^{4/3}-h^2O(n^{4/3}).$$
Dividing throughout by $n^{4/3}$ and taking limit as $n$ goes to infinity we get 
$${w(0)-w(h)}=h+O(h^2).$$
Thus dividing by $h$ and sending $h$ to $0$ we are done.
\end{proof}

\subsection{Proof of Theorem \ref{conditionalscalingthm}}
From Lemmas \ref{valuea} and \ref{valueb} and using the symmetry of $w$ about the origin we get $$w(x)=\left(\left(\frac{9\pi}{32}\right)^{1/12}-\left(\frac{32}{9\pi}\right)^{1/4}\frac{|x|}{4}\right)^4$$ on the region of support.
Rearranging we get 
\begin{equation} \label{conjecturedscalinglimit}
w(x) = \begin{cases} \frac{1}{72\pi}\left ((18\pi)^{1/3} -|x| \right)^4, & |x|< (18\pi)^{1/3} \\
				0 & |x| \geq (18\pi)^{1/3}. \qedhere \end{cases}
\end{equation}

\section{Open Questions}
%!TEX root = oil-and-water.tex
\label{s.open}

Conjecture~\ref{ass1} is an obvious target. In this concluding section we collect some additional open questions.

\subsection{Location of the rightmost particle}

For the oil and water process with $n$ particles of each type started at the origin $\Z$, let $R_n$ be the location of the rightmost particle upon fixation.  Is the sequence of random variables $R_n / n^{1/3}$ tight?  Does it converge in distribution to a constant?  If it does, then Theorem \ref{conditionalscalingthm} suggests that the limit should be at least $(18 \pi)^{1/3}$ (and perhaps equal to this value).

\subsection{Order of the variance}

We believe that the standard deviation of the odometer $u$ is of order $n^{7/6}$ in the bulk. Note that Conjecture~\ref{ass1} asserts something weaker, namely $o(n^{4/3})$.

Here is a heuristic argument for the exponent $7/6$.  
The total number of particle exits from $x$ is $2u(x)$; let $N_x$ be the total number of particle \emph{entries} to $x$.  Equating entries minus exits with the number of particles left behind, we find that
 	\begin{equation} \label{e.repn} \Delta u(x) := u(x-1) + u(x+1) - 2u(x) = Z_1(x) +|Z_2(x)| - 2n \delta_0(x). \end{equation}
where $Z_1(x) = u(x-1) + u(x+1) - N_x$, and $Z_2(x)$ is the signed count of the number of particles remaining at $x$ in the final state (counting oil as positive, water as negative). 	
Both $Z_1(x)$ and $Z_2(x)$ are expressable as sums of indpendent indicators involving the stack elements at $x \pm 1$.  
\old{%%% the details:
Recall the stacks mentioned in \eqref{set1}. We claim the following:
\begin{align*}
2n+\Delta u(0)&=\sum_{i=1}^{u(-1)}\mathbf{1}({X}^{-1}_{i}=1,{Y}^{-1}_{i}=1)+\sum_{i=1}^{u(1)}\mathbf{1}({X}^{1}_{i}=-1,{Y}^{1}_{i}=-1)\\ &-\sum_{i=1}^{u(-1)}\mathbf{1}({X}^{-1}_{i}=-1,{Y}^{-1}_{i}=-1)-\sum_{i=1}^{u(1)}\mathbf{1}({X}^{1}_{i}=1, {Y}^{1}_{i}=1)\\& +\left|\sum_{i=1}^{u(-1)}\mathbf{1}({X}^{-1}_{i}=1,{Y}^{-1}_{i}=-1)+\sum_{i=1}^{u(1)}\mathbf{1}({X}^{1}_{i}=-1,{Y}^{1}_{i}=1)\right.\\ &
\left. -\sum_{i=1}^{u(-1)}\mathbf{1}({X}^{-1}_{i}=-1,{Y}^{-1}_{i}=-1)-\sum_{i=1}^{u(1)}\mathbf{1}({X}^{1}_{i}=-1,{Y}^{1}_{i}=1)\right|
\end{align*}

To see this notice that
\begin{align*}
u(-1)+u(1) &+& \sum_{i=1}^{u(-1)}\mathbf{1}({X}^{-1}_{i}=1,{Y}^{-1}_{i}=1)+\sum_{i=1}^{u(1)}\mathbf{1}({X}^{1}_{i}=-1,{Y}^{1}_{i}=-1)\\ & -&\sum_{i=1}^{u(-1)}\mathbf{1}({X}^{-1}_{i}=-1,{Y}^{-1}_{i}=-1)-\sum_{i=1}^{u(1)}\mathbf{1}({X}^{1}_{i}=1, {Y}^{1}_{i}=1)
\end{align*}
is the total number of particles coming in to the origin from the neighbors. The term
\begin{align*}
&  \bigl|\sum_{i=1}^{u(-1)}\mathbf{1}({X}^{-1}_{i}=1,{Y}^{-1}_{i}=-1)+\sum_{i=1}^{u(1)}\mathbf{1}({X}^{1}_{i}=-1,{Y}^{1}_{i}=1) \\&
-\sum_{i=1}^{u(-1)}\mathbf{1}({X}^{-1}_{i}=-1,{Y}^{-1}_{i}=-1)-\sum_{i=1}^{u(1)}\mathbf{1}({X}^{1}_{i}=-1,{Y}^{1}_{i}=1)\bigr|
\end{align*}
is the absolute difference between the number of oil and water particles at the origin at the end of the process i.e. number of particles finally at the origin. Also recall that the initial number of particles at the origin was $2n.$
}%%%
The limits of summation are $u(x \pm 1)$.  Assuming Conjecture~\ref{ass1} and arguing as in Lemma~\ref{l.clt}, we can replace the limits of summation by their expected values $\tilde{u}(x\pm 1)$, incurring only a small error.  The resulting sums $\tilde Z_1$ and $\tilde Z_2$ are asymptotically normal with mean zero and variance of order $n^{4/3}$ (assuming $x$ is in the bulk, $|x|< ((18\pi )^{1/3}-\eps) n^{1/3}$).  Moreover, the function $\tilde{Z}_1 + |\tilde{Z}_2|$ is \emph{$2$-dependent}: its values at $x$ and $y$ are independent if $|x-y|>2$. By summation by parts,
 	\[ u(x) = \sum_{j=1}^{\infty} j\Delta u(x+j). \]
Since most of the support of $u$ is on an interval of length $O(n^{1/3})$, truncating this sum at $Cn^{1/3}$ for a large constant $C$ should not change its variance by much.
Replacing $\Delta u$ by its approximation $\tilde Z_1 + |\tilde Z_2|$  and using the $2$-dependence, we arrive at
	\[ \Var u(x) = \sum_{j=1}^{Cn^{1/3}} j^2 O(n^{4/3}) = O(n^{7/3}). \]

\subsection{Conjectured exponents in higher dimensions}

For the oil and water model in $\Z^d$ starting with $n$ oil and $n$ water particles
at the origin, we believe that the typical order of the odometer (away from $0$ and the
boundary) is $n^{4/(d+2)}$ and the radius of the occupied cluster is of
order $n^{1/(d+2)}$. The reason is by analogy with Section \ref{sect:final_step}: if $w : \R^d \to \R$ solves the PDE
\begin{equation}
\label{e.thepde}
\Delta w = -\delta_0 + \sqrt{\frac{2}{\pi}w}
\end{equation}
then its rescaling
\[ v(x) = t^4 w(x/t) \]
satisfies
\[ \Delta v = -t^{d+2}\delta_0 + \sqrt{\frac{2}{\pi}v}. \]
If the odometer for $n$ particles has a scaling limit $w$ that satisfies \eqref{e.thepde}, then $v$ is the
scaling limit of the odometer for $t^{d+2}n$ particles. So increasing the number of particles a factor of $t^{d+2}$ increases the radius by by a factor of $t$ and the odometer by a factor of $t^4$. This motivates the following conjecture.

\begin{conjecture}
Let $u$ be the odometer for the oil and water model started from n particles of each type at the origin in $\Z^d$. There exists a deterministic function $w: \R^d \to \R$ such that for all $\xi \in \R - \{0\}$ we have almost surely,
	\[ \frac{u(\lfloor n^{1/(d+2)} \xi \rfloor)}{n^{4/(d+2)}}  \to w(\xi). \]
Moreover, $w$ is rotationally symmetric, twice differentiable on $\R^d - \{0\}$ and satisfies
	\[ \Delta w = \sqrt{\frac{2}{\pi}w} \]
on $\R^d - \{0\}$ and $\lim_{\xi \to 0} \frac{w(\xi)}{g(\xi)} = 1$ where $g$ is the Green function for the Laplacian on $\R^d$.
\end{conjecture}

The fourth power scaling is reflected in the even spacing between contour lines of the odometer function in Figure~\ref{f.2dcontours}.

\begin{figure}
\centering
\includegraphics[width=.48\textwidth]{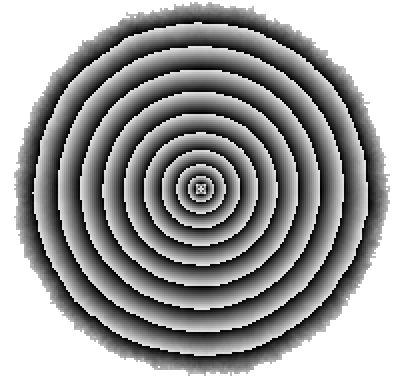}
\caption{Contour lines of the odometer function $u$ of the oil and water model in $\Z^2$ with $n=2^{22}$ particles of each type started at the origin. Each site is shaded according to the fractional part of $\frac15 u^{1/4}$.}
\label{f.2dcontours}
\end{figure}

\newpage

\appendix
\section{Appendix}\label{porse}
%!TEX root = oil-and-water.tex

\subsection{Concentration Estimates}\label{section:ConcentrationEstimates}

\begin{lemma}\label{lemma:concentration} Suppose for all $x\in \mathbb{Z}$
we have \[
\begin{array}{c}
X^{x}_{1},Y^{x}_{1}\\[1.1ex]
X^{x}_2,Y^{x}_2\\[1.1ex]
\vdots\\[1.1ex]

\end{array} 
\] a sequence of independent uniform $\pm1$ valued random variables. The sequences across $x$ are also independent of each other.
Then there exists constants $C,C',\gamma >0$ such that for $n$ large enough, with probability at least $1-C\exp(-C'n^{\gamma})$  for all  $ k>\sqrt{n}$  and $- n^{5}<j< n^{5}$ we have
\begin{itemize}
\item[(i)] $\displaystyle \left|\sum_{i=1}^{k}\mathbf{1}_{(X^{j}_i=1)}- k/2\right| < k^{1/2+\epsilon}$;
\item[(ii)] $\displaystyle \left|\sum_{i=1}^{k}\mathbf{1}_{(Y^{j}_{i}=1)}- k/2\right| < k^{1/2+\epsilon}$;
\item[(iii)] $\displaystyle \left|\sum_{i=1}^{k}\mathbf{1}_{(X^{j}_{i}=1)}\mathbf{1}_{(Y^{j}_{i}=1)}- k/4\right| < k^{1/2+\epsilon}$;
\item[(iv)] $\displaystyle \left|\sum_{i=1}^{k}\mathbf{1}_{(X^{j}_{i}=-1)}\mathbf{1}_{(Y^{j}_{i}=-1)}- k/4\right| < k^{1/2+\epsilon}$.
\end{itemize}
\end{lemma}
\begin{proof} Proof follows by standard bounds from Azuma-Hoeffding's inequality for Bernoulli random variables and union bound over $k\ge \sqrt{n}$ followed by $j\in [-n^{5},n^5]$. 
\end{proof}

\subsection{Proof of Lemma \ref{mean2}} Let us define the truncated variable $$Y = M({n^{\frac{4}{3}}}) \mathbf{1}\left(M({n^{\frac{4}{3}}}) < n ^ {\frac{2}{3}+\e}\right)$$ for some small but a priori fixed $\e$. Let $Y_i$ be iid copies of $Y.$
Now by using Azuma's inequality,
$$\P\left(\sum_{i=1}^{n^{1/3}} (Y_i-\E(Y_i)) > t\right)\le e^{-\frac{t^2}{n^{1/3}n^{4/3+2\e}}}.$$ 

Taking $t^2=n^{5/3+3\e}$ we get that 
$$\P\left(\sum_{i=1}^{n^{1/3}} (Y_i-\E(Y_i)) > t\right)<e^{-n^{\e}}$$
Now by $(iv)$ Lemma \ref{mean1} $\E(Y)=O(n^{2/3})$. Thus $$n^{1/3}\E(Y)+t<Dn$$  for some large $D$ as  $t= n^{5/6+2\e}$.  Hence 
$$\P\left(\sum_{i=1}^{n^{1/3}} Y_i>Dn\right)\le e^{-n^{\e}}$$
This implies that 
$$\P\left(\sum_{i=1}^{n^{1/3}} M^{i}({n^{\frac{4}{3}}})>Dn\right)\le Ce^{-n^{\gamma}}$$ since 
by $(i)$ Lemma \ref{mean1} and union bound, there exists a positive constant $c>0$ such that
$$\mathbb{P}\left(\exists\,\, 1\le i \le n^{1/3}\,\,\text {such that }Y_i \neq M^{i}\right)\le e^{-n^{c}}.$$
\qed

\subsection{Proof of Lemma \ref{momentbound}}
We use the variables $E_i$ defined in the statement of Lemma \ref{waitingtimes}. 
Let $$Y=\sum_{i=1}^{\tau'}E_i$$
where $\tau'$ was defined in \eqref{rwhittime}.
As mentioned in proof of Lemma  \ref{aprioridom}
by Lemmas \ref{firingrule} and \ref{waitingtimes} $\tau$ is stochastically dominated by $Y$.
Thus $$\E\left[\tau\mathbf{1}(\tau>n^5)\right] \leq \E\left[Y\mathbf{1}(Y>n^5)\right].$$
Hence to prove the lemma it suffices to show the right hand side is $O(1)$. Now
 \begin{eqnarray}
\E\left[Y\mathbf{1}(Y>n^5)\right]&\le & \E\left[Y\mathbf{1}\bigl(\tau'>n^{\frac{5}{2}}\bigr)\right]+\E\left[Y\mathbf{1}(\tau'\le n^{\frac{5}{2}})\mathbf{1}(Y>n^5)\right]\\
&\le &\E\left[Y\mathbf{1}(\tau'>n^{\frac{5}{2}})\right]+\E\Bigl[\sum_{i=1}^{n^{\frac{5}{2}}}E_i\mathbf{1}(\sum_{i=1}^{n^{\frac{5}{2}}}E_i\ge n^5)\Bigr]\\ 
\label{imp12}
&\le & \E\left[Y\mathbf{1}(\tau'>n^{\frac{5}{2}})\right] +\E\Bigl[\sum_{i=1}^{n^{\frac{5}{2}}}E_i\Bigr]\Bigl[\sum_{i=1}^{n^{\frac{5}{2}}}\mathbf{1}(E_i\ge n^{\frac{5}{2}})\Bigr] 
\end{eqnarray}
where the last inequality follows from the easy fact
$$
\mathbf{1}\left(\sum_{i=1}^{n^{\frac{5}{2}}}E_i\ge n^5\right)\le \sum_{i=1}^{n^{\frac{5}{2}}}\mathbf{1}(E_i\ge n^{\frac{5}{2}}).
$$
We use the following tail estimate for $E_1$ and $\tau'$: there exists a constant $c<1$ such that for $k\ge n^2$ 
\begin{equation}\label{tailest4}
\max(\mathbb{P}(\tau'\ge k),\mathbb{P}(E_1\ge k))\le (1-c)^{\lfloor\frac{k}{n^2}\rfloor}
\end{equation}
which easy follows from the fact that starting  from any point in $[-2n,2n]$ there exists a constant chance $c$ for the random walk to exit the interval in the next $n^2$ steps. 
Using \eqref{tailest4}, independence of $\tau'$, $E_i's$, the theorem now follows from \eqref{imp12}.
The details are omitted.  
\qed
\subsection{Proof of Corollary \ref{boundexp}}
 The proof follows from the following observation:
\begin{equation}\label{keybound1}
u(x)\le Cn^{4/3}\mathbf{1}(\m)+n^{5}\mathbf{1}(\m^{c})\mathbf{1}(\tau\le n^5)+\tau\mathbf{1}(\tau\ge n^5)\mathbf{1}(\m^{c})
\end{equation}
where  $C$ is the constant appearing in the statement of Theorem \ref{lemmaJ}, $\tau$ is defined in \eqref{stoptimedef} and $\m$ is defined in \eqref{mainevent}.
The first term follows from the definition of $\m$. For the second and third term we use the trivial bound that $$u(x)\le \tau.$$
Taking expectation we get 
\begin{eqnarray*}
\tilde u(x)\le Cn^{4/3}+n^{5}\pr(\m^c)+\E(\tau\mathbf{1}(\tau\ge n^5)).
\end{eqnarray*}
The last two terms are $O(1)$ by Theorem \ref{lemmaJ} and Lemma \ref{momentbound} respectively.
Hence we are done.
\qed

\section*{Acknowledgments}
We are grateful to Alexander Holroyd and Yuval Peres for valuable discussions, and to Deepak Dhar for bringing reference \cite{2comp} to our attention. We thank Wilfried Huss and Ecaterina Sava-Huss for a careful reading of an earlier draft, and for detailed comments which improved the paper.

\nocite{}
\bibliography{bibliography}{}
\bibliographystyle{plain}

\end{document}